\documentclass[12pt]{amsart}
\usepackage{amsmath,amsfonts,amssymb,amsthm,} 
\usepackage[foot]{amsaddr}
\usepackage{mathtools} 
\usepackage{bm} 
\usepackage{xfrac} 
\usepackage{scalerel} 
\usepackage{stmaryrd} 
\usepackage{pgf,tikz}%
\usetikzlibrary{arrows,cd}%
\usepackage[hidelinks]{hyperref}%
\usepackage{graphicx,wrapfig}%
\usepackage{multirow}
\usepackage{paralist}
\usepackage[inline]{enumitem}
\usepackage[margin=1in]{geometry}
\usepackage{caption, subcaption}
\colorlet{RED}{red}

\newtheorem{theorem}{Theorem}
\newtheorem{prop}[theorem]{Proposition}
\newtheorem{lemma}[theorem]{Lemma}
\newtheorem{cor}[theorem]{Corollary}
\newtheorem{definition}[theorem]{Definition}

\newtheorem{remark}[theorem]{Remark}

\newtheorem{question}[theorem]{Question}

\DeclareMathAlphabet{\mathup}{OT1}{\familydefault}{m}{n}%
\newcommand{\RR}{\ensuremath{\mathbb{R}}} 
\newcommand{\CC}{\ensuremath{\mathbb{C}}} 
\newcommand{\HH}{\ensuremath{\mathbb{H}}} 

\DeclareMathOperator{\MCG}{MCG}
\DeclareMathOperator{\Homeo}{Homeo}

\newcommand{\leftQ}[2]{\left.\raisebox{-.2em}{$#2$}\middle\backslash\raisebox{.2em}{$#1$}\right.}
\newcommand{\rightQ}[2]{\left.\raisebox{.2em}{$#1$}\middle/\raisebox{-.2em}{$#2$}\right.}

\title{A Moduli Space of Marked Hyperbolic Structures for Big Surfaces}
\author{Chaitanya Tappu}
\address{Department of Mathematics\\ 310 Malott Hall, Cornell University, Ithaca, NY 14853}
\email{tappu@math.cornell.edu}
\date{\today}

\begin{document}
    \begin{abstract}
        We introduce the moduli space of marked, complete, Nielsen-convex hyperbolic structures on a surface of negative, but not necessarily finite, Euler characteristic.
        The emphasis is on infinite type surfaces, the aim being to study mapping class groups of infinite type surfaces via their action on this marked moduli space.
        We define a topology on the marked moduli space and prove that it reduces to the usual Teichm\"uller space for finite type surfaces.
        We also prove that the action of the mapping class group on this marked moduli space is continuous.
    \end{abstract}
    
    \maketitle
    
    \section{Introduction}
    
    The Teichm\"uller space can be thought of as the moduli space of either marked Riemann surface structures or marked complete hyperbolic structures on a fixed finite type surface.
    The two viewpoints are equivalent due to the uniformisation theorem and the fact that the isometries of the hyperbolic plane are exactly the biholomorphisms.
    The mapping class group of the surface acts on the Teichm\"uller space by changing the marking.
    This action has been studied classically, with important consequences for the mapping class group such as the Nielsen--Thurston classification of mapping classes, the geometric classification of mapping tori, the solution to the Nielsen realisation problem, etc (see \cite{thurston_1988_geometry_dynamics_diffeomorphisms_surfaces}, \cite{farb_margalit_2011_primer}, \cite{hubbard_2006_teichmueller}, \cite{hubbard_2016_teichmueller2}, \cite{hubbard_2022_teichmueller3}, \cite{thurston_1998_hyperbolic_manifolds_fiber_circle}, \cite{kerckhoff_1983_nielsen_realization}).
    Let $S$ be a surface with $-\infty \le \chi (S) < 0$.
    In this paper, we introduce the moduli space $\mathcal T (S)$ of marked, complete, Nielsen-convex hyperbolic structures on $S$, or `marked moduli space' for short.
    Here and henceforth in this paper, the phrase `complete' hyperbolic surface always means a `geodesically complete' hyperbolic surface.
    We are especially interested in infinite type surfaces $S$, that is, the surfaces for which $\chi (S) = -\infty$.
    We equip the marked moduli space with an action of the mapping class group $\MCG (S)$ of the surface $S$, which is a topological group.
    The main result of this paper is that the $\MCG (S)$ action on $\mathcal T (S)$ is continuous.
    We also prove that the space $\mathcal T (S)$ reduces to the usual Teichm\"uller space in case $S$ is a finite type surface, that is, a surface for which $\chi (S)$ is negative and finite.
    
    In fact, Teichm\"uller spaces have already been defined and studied even for infinite type surfaces (see \cite[Chapter 6]{hubbard_2006_teichmueller}).
    However, these are defined for a Riemann surface $X$ rather than a topological surface.
    In particular, the Teichm\"uller space depends on the quasiconformal class of the `basepoint' Riemann surface structure $X$.
    Consequently they admit natural actions by only the quasiconformal mapping class group $\operatorname{QMCG} (X)$, which is a proper subgroup of the full mapping class group $\MCG(X)$ of the underlying topological surface.
    On the other hand, the full group $\MCG(S)$ acts on $\mathcal T (S)$, so we expect $\mathcal T (S)$ to be a useful space for studying $\MCG(S)$.
    Another reason to consider $\mathcal T (S)$ is a theorem of Thurston (\cite[Corollary 5.4]{thurston_1986_earthquakes}) that there exists an essentially unique earthquake between any two relative hyperbolic structures on a complete hyperbolic surface. In fact, our definition of $\mathcal T (S)$ is inspired by this paper.
    
    Let $S$ be a connected, oriented surface with negative Euler characteristic, or equivalently, a nonabelian fundamental group.
    The Euler characteristic need not be finite; our emphasis is on \emph{infinite type surfaces}, those whose Euler characteristic is $-\infty$, or equivalently, surfaces whose fundamental group is not finitely generated.
    We assume $S$ does not have a boundary.
    In this article, all surfaces under consideration will be oriented and all homeomorphisms will be orientation preserving, and we suppress mention of orientation.
    We now define the set $\mathcal T (S)$.
    The topology of $\mathcal T (S)$ is defined in Definition \ref{def:marked_moduli_space_topology} via homeomorphisms at infinity.
    In Theorem \ref{thm:embedding_character_space} we show that the same topology comes via characters of holonomy representations.
    \begin{definition}[Set of marked, complete, Nielsen-convex hyperbolic structures]\label{def:marked_moduli_space}
        \begin{equation}\label{eq:def_marked_moduli_space}
            \mathcal T (S) = \rightQ{\left\{(X, f) \middle\vert\begin{aligned}X &\text{ is a complete Nielsen-convex hyperbolic surface}\\ f& : S \to X \text{ is a homeomorphism} \end{aligned}\right\}}{\sim}
        \end{equation}
        where $(X_1, f_1) \sim (X_2, f_2)$ if there is an isometry $\varphi : X_1 \to X_2$ homotopic to $f_2 \circ f_1^{-1} $.
    \end{definition}
    Here we have adopted the term \emph{Nielsen-convex} from work of Alessandrini, Liu and others (\cite[Definition 4.3]{alessandrini_liu_2011_fenchel_nielsen}).
    They define a notion of Nielsen-convexity for hyperbolic surfaces which may or may not be complete.
    However, we are interested only in complete hyperbolic surfaces.
    So instead of recalling the general definition, we simply state equivalent ways of thinking of Nielsen-convexity for complete hyperbolic surfaces.
    \begin{prop}
        \label{prop:nielsen_convex_alternatives}
        Let $X$ be a complete hyperbolic surface.
        Then the following are equivalent.
        \begin{enumerate}
            \item $X$ is Nielsen-convex.
            \item The convex core $C (X)$ of $X$ equals $X$.
            \item $X$ can be constructed by gluing hyperbolic pairs of pants (possibly with cusps) along their boundary components.
            \item The limit set of the action of $\pi_1 (X)$ on the universal cover $\HH^2$ is the entire circle at infinity.
            \item $X$ is isometric to $\leftQ{\HH^2}{\Gamma}$ where $\Gamma$ is a torsion-free Fuchsian group of the first kind.
            \item The ideal boundary $I (X)$ of $X$ is empty.
            \item $X$ has no visible ends.
        \end{enumerate}
    \end{prop}
    The equivalence of the first and the second conditions is \cite[Proposition 4.6]{alessandrini_liu_2011_fenchel_nielsen}, and that of the first and the third conditions is part of \cite[Theorem 4.5]{alessandrini_liu_2011_fenchel_nielsen}.
    The equivalence of the second and the fourth conditions follows easily from the definition of the convex core, which we recall.
    The complete hyperbolic surface $X$ is isometric to $\leftQ{\HH^2}{\Gamma}$ for some torsion-free Fuchsian group $\Gamma$.
    Let $\Lambda \subset \partial \HH^2 = S^1$ be the limit set of $\Gamma$, and $CH (\Lambda)$ be the convex hull of $\Lambda$.
    Then the convex core is $C (X) = \leftQ{CH (\Lambda)}{\Gamma}$.
    Thus $C(X) = X$ if and only if $CH (\Lambda) = \HH^2$, which is equivalent to $\Lambda = S^1$.
    Note that in \cite{alessandrini_liu_2011_fenchel_nielsen}, the hyperbolic pair of pants with three cusps is treated separately due to technical reasons concerning their definition of Nielsen-convexity.
    However, since it satisfies the second condition, we will call it Nielsen-convex.
    The fourth condition above says that $\pi_1 (X)$, viewed as a torsion-free Fuchsian group $\Gamma$, is of the first kind, and $X$ is isometric to $\leftQ{\HH^2}{\Gamma}$, which is exactly the fifth assertion.
    The equivalence of the fifth and sixth conditions follows easily from the definition of the ideal boundary, which is $I (X) = \leftQ{(S^1 \setminus \Lambda)}{\Gamma}$ (see \cite[Section 3.7]{hubbard_2006_teichmueller}).
    For a discussion of visible ends, see \cite[Section 2]{basmajian_saric_2019_geodesically_complete}.
    
    The set $\mathcal T (S)$ is indeed nonempty.
    Take a topological pants decomposition $\mathcal P$ of $S$, where a pant is a surface of zero genus, $b$ boundary components and $n$ punctures with $b + n = 3$.
    For each pant in the scheme $\mathcal P$, consider a hyperbolic pair of pants (with cusps at each puncture).
    The geometry of the hyperbolic pairs of pants is chosen such that the cuff lengths of two boundary components that get glued in the scheme $\mathcal P$ are equal.
    This allows the hyperbolic pairs of pants to be glued according to the scheme $\mathcal P$ to produce a hyperbolic surface $X$ homeomorphic to $S$ via a marking homeomorphism $f : S \to X$.
    If the cuff lengths are all bounded above, then the injectivity radius of points in $X$ is bounded below, and therefore $X$ is complete.
    Further, since $X$ is a union of hyperbolic pairs of pants, Proposition \ref{prop:nielsen_convex_alternatives} asserts that $X$ is a complete, Nielsen-convex hyperbolic surface.
    Therefore $[X, f] \in \mathcal T (S)$, and so $\mathcal T(S)$ is nonempty.
    In fact, even if the lengths of the cuffs are not bounded above, \cite[Theorem 5.1]{basmajian_saric_2019_geodesically_complete} asserts that the hyperbolic pairs of pants may be glued with particular choices of twists in such a way that the resulting surface $X$ is geodesically complete.

    We remark that if $S$ is a closed surface, then the above set $\mathcal T (S)$ reduces, as a set, to the usual Teichm\"uller space (also see Corollary \ref{cor:marked_moduli_space_reduces_to_teichmuller_space}).
    This is because a closed hyperbolic surface has empty ideal boundary, and does not have visible ends (or any ends at all).
    The above is also true for finite type punctured surfaces.
    Indeed, if $S$ is a finite type punctured surface, then in the usual definition of Teichm\"uller space, the hyperbolic surface $X$ is constrained to have finite area.
    For finitely generated Fuchsian groups, the properties of having finite coarea and having full limit set are equivalent.
    Thus our definition is an extension of the usual Teichm\"uller space to infinite type surfaces.
    
    The \emph{mapping class group} of the surface $S$ is $\MCG (S) = \rightQ{\Homeo^+ (S)}{\Homeo_0 (S)}$.
    The group $\Homeo^+ (S)$ is a topological group with the compact-open topology, which agrees, for any metric on $S$, with the topology of uniform convergence on compact subsets.
    The subgroup $\Homeo_0 (S)$ of homeomorphisms isotopic to the identity is a closed subgroup.
    Hence the quotient $\MCG(S)$ is a topological group with the quotient topology (see \cite[Section 2.3]{aramayona_vlamis_2020_big_mapping_class}).
    It is clear that $\MCG(S)$ acts on $\mathcal T (S)$.
    If $\psi$ is a self-homeomorphism of $S$, the mapping class $[\psi]$ acts on a marked hyperbolic structure $[X, f]$ by changing the marking from $f$ to $f \circ \psi^{-1}$.
    We denote the action of the mapping class $[\psi]$ by $A_{[\psi]}$, so that $A_{[\psi]}$ is a function from $\mathcal T (S)$ to itself.
    \begin{prop}
        There is a well defined group action $A : \MCG (S) \times \mathcal T (S) \to \mathcal T (S)$ given by
        \begin{equation} \label{eq:action}
            A ([\psi], [X, f]) = [X, f \circ \psi^{-1}]
        \end{equation}
    \end{prop}
    \begin{proof}
        We have to show that the action function $A$ is independent of the choices of the hyperbolic surface $X$, the marking map $f$ and the representative homeomorphism $\psi$.
        This follows easily from the diagram in Figure 1
        , which homotopy commutes.
        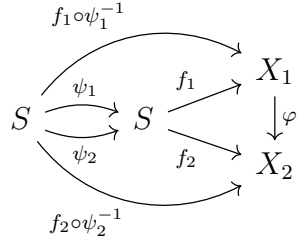
\begin{figure}\label{fig:action_well_defined}
            \centering
            \tikzcdset{row sep/null/.initial=0em}
            $\begin{tikzcd}[row sep = null]
                & & X_1 \arrow[dd,"\varphi"] \\
                S \arrow[r, bend left=20, "\psi_1" pos=0.55] \arrow[r, bend right=20, "\psi_2"' pos=0.55] \arrow[rru, bend left=40, "f_1 \circ \psi_1^{-1}"] \arrow[rrd, bend right=40, "f_2 \circ \psi_2^{-1}"'] & S \arrow[ru, "f_1"] \arrow[rd, "f_2"'] & \\
                & & X_2
            \end{tikzcd}$
            \caption{The action of $\MCG(S)$ on $\mathcal T (S)$ is well defined.}
        \end{figure}
    \end{proof}
    
    The main result of this paper is the following:
    \begin{theorem}\label{thm:main}
        The set $\mathcal T (S)$ has a geometrically defined topology (see Definition \ref{def:marked_moduli_space_topology}), which agrees with the usual topology on the Teichm\"uller space when $S$ is of finite type (see Corollary \ref{cor:marked_moduli_space_reduces_to_teichmuller_space}).
        With respect to this topology, the action function $A$ is continuous (see Theorem \ref{thm:action_continuous}) and $\MCG (S)$ acts on $\mathcal T (S)$ by homeomorphisms (see Corollary \ref{cor:action_by_homeomorphisms}).
    \end{theorem}
    
    Having defined the topology of the marked moduli space, it is natural to ask what geometry it admits.
    In particular, we seek a metric on the marked moduli space so that the mapping class group acts isometrically on the marked moduli space.
    \begin{question}
        Is there a natural (geometrically defined) $\MCG (S)$-invariant metric on $\mathcal T (S)$?
    \end{question}
    
    \textbf{Organisation of the paper:}
    In Section \ref{sec:notation}, we fix notation and recall some basic facts of hyperbolic geometry and algebraic topology.
    In Section \ref{sec:homeomorphism_at_infinity}, we prove a key tool about homeomorphisms at infinity that will be used repeatedly.
    In Section \ref{sec:topology}, we define the topology on the marked moduli space using homeomorphisms at infinity.
    In Section \ref{sec:continuity}, we prove that the action of the mapping class group on the marked moduli space is continuous.
    In Section \ref{sec:embedding}, we prove that the topology on the marked moduli space agrees with the topology coming from injecting it into the $\mathup{PSL} (2, \RR)$-character space of the fundamental group. It follows that the topology of the marked moduli space reduces to the usual topology of Teichm\"uller space in the case of finite type surfaces.
    
    \textbf{Acknowledgements:} I would like to thank my advisor Prof.~Jason Manning for his constant support, encouragement and many helpful conversations about the subject matter.
    I would also like to thank Jason Manning and Olu Olorode for a careful reading of the manuscript, and would like to thank Yassin Chandran and Nick Vlamis for helpful conversations.
    
    \section{Notation and Background}
    \label{sec:notation}
    The \emph{hyperbolic plane} may be described as either the upper half plane or the unit disk in $\CC P^1 = \CC \cup \{\infty\}$, and is equipped with the Poincar\'e metric and the usual orientation.
    These two models of the hyperbolic plane are isometric.
    We denote the hyperbolic plane by $\HH^2$.
    It has a \emph{boundary at infinity} which is $\RR P^1 = \RR \cup \{\infty\}$ in the upper half plane model or $S^1$ in the unit disk model.
    We denote the boundary of the hyperbolic plane by $S^1$.
    $S^1$ inherits an orientation as the boundary of $\HH^2 \cup S^1$, and also possesses a circular order, so that $(0, 1, \infty)$ is a positively oriented triple of points.
    
    We denote the group of orientation preserving isometries of $\HH^2$ by $\mathup{PSL} (2, \RR)$, following the description of isometries of the upper half plane.
    Each isometry of $\HH^2$ extends to an orientation preserving homeomorphism of $S^1$.
    In fact, $\mathup{PSL} (2, \RR)$ acts freely and transitively on positively oriented triples of points in $S^1$.
    For a positively oriented triple $(a, b, c)$ in $S^1$, we denote by $M (a, b, c)$ the unique element of $\mathup{PSL} (2, \RR)$ that maps $0$ to $a$, $1$ to $b$ and $\infty$ to $c$.
    In particular, any isometry of $\HH^2$ is determined by its homeomorphism at infinity, so $\mathup{PSL} (2, \RR)$ is a subgroup of $\Homeo^+ (S^1)$, which is the group of orientation preserving homeomorphisms of the circle.
    We will freely use the same symbol to denote an isometry of $\HH^2$ as well as the induced homeomorphism at infinity.
    $\Homeo^+ (S^1)$ is a topological group with the compact-open topology, which agrees, for any metric on $S^1$, with the topology of uniform convergence.
    With this topology, $\mathup{PSL} (2, \RR)$ is a closed and embedded subgroup of $\Homeo^+ (S^1)$.
    
    Each complete, oriented hyperbolic surface $X$ has a (Riemannian, orientation preserving) universal cover $p : \HH^2 \to X$, and any other such cover is of the form $p \circ \sigma$ for some $\sigma \in \mathup{PSL} (2, \RR)$.
    We denote the deck group of the universal cover $p$ by $\Gamma_X$, which is a torsion-free subgroup of $\mathup{PSL} (2, \RR)$.
    In terms of the deck group, we can express $X$ as $\leftQ{\HH^2}{\Gamma_X}$.
    Given the cover $p$, any basepoint $x \in X$ and any choice of its lift $\tilde x \in \HH^2$, or equivalently, given any pointed universal cover $p : (\HH^2, \tilde x) \to (X, x)$, there is an isomorphism $\varphi : \pi_1 (X, x) \to \Gamma_X$, which is known as the \emph{holonomy representation} (see \cite[Section 1.3]{hatcher_2002_algebraic_topology}).
    If $\alpha$ is an oriented closed curve in $X$ based at $x$ and $\tilde \alpha$ is a (bi-infinite) lift of $\alpha$ passing through $\tilde x$, then the deck transformation $\varphi [\alpha]$ acts on $\tilde \alpha$ by translation.
    Note that the isomorphism $\varphi$ depends on the choice of lift $\tilde x$.
    If $\tilde x' \in \HH^2$ and $x' = p (\tilde x')$ are different basepoints with the corresponding holonomy representation $\varphi' : \pi_1 (X, x') \to \Gamma_X$, then the two holonomy representations satisfy the relation $\varphi' = \varphi \circ c_\beta$.
    Here $\beta$ is the projection $p (\tilde \beta)$ of a continuous path $\tilde \beta$ from $\tilde x$ to $\tilde x'$, and $c_\beta$ is `conjugation by $\beta$', that is, $c_\beta [\alpha] = [\beta \cdot \alpha \cdot \overline \beta]$ for all $[\alpha] \in \pi_1 (X, x')$.
    If $X, Y$ are complete hyperbolic surfaces with universal covers $p_X : \HH^2 \to X$ and $p_Y : \HH^2 \to Y$, then an arbitrary homeomorphism $f : X \to Y$ lifts to a homeomorphism $\tilde f : \HH^2 \to \HH^2$.
    This induces an isomorphism $f_* : \Gamma_X \to \Gamma_Y$ between the two deck groups, which is `conjugation by $\tilde f$'.
    Note that the isomorphism $f_*$ depends on the choice of the lift $\tilde f$, but we denote it by $f_*$, suppressing the lift from the notation.
    If $\varphi_X$ and $\varphi_Y$ (with respect to basepoints $\tilde y = \tilde f (\tilde x)$ and $y = p_Y (\tilde y)$) are the holonomy representations of $X$ and $Y$ respectively, then we have $f_* (\varphi_X [\alpha]) = \varphi_Y (f_*[\alpha])$.
    Here the $f_*$ on the left hand side is the isomorphism between the deck groups, whereas on the right hand side, $f_*$ is the $\pi_1$ functor.
    In other words, if $\gamma \in \Gamma_X$ is the holonomy around the oriented closed $\alpha \subset X$, then $f_* (\gamma) \in \Gamma_Y$ is the holonomy around the oriented closed curve $f (\alpha)$.
    Another fact we recall is that the set of all lifts of $f$ is $\tilde f \circ \Gamma_X$, which also equals $\Gamma_Y \circ \tilde f$.
    
    As a torsion-free discrete subgroup of $\mathup{PSL} (2, \RR)$, $\Gamma_X$ cannot have any elliptic isometries of $\HH^2$.
    Thus all the elements of $\Gamma_X$ are either hyperbolic or parabolic isometries of $\HH^2$.
    For a hyperbolic isometry $\gamma \in \mathup{PSL} (2, \RR)$, we denote by $\gamma_\infty$, the unique attracting fixed point of $\gamma$ on $S^1$, also known as its \emph{sink}.
    If $\gamma$ is the holonomy around an oriented closed curve $\alpha$ in $X$, then the the (bi-infinite) lift $\tilde \alpha$ passing through $\tilde x$ joins the point $(\gamma^{-1})_\infty$ to the point $\gamma_\infty$.
    We denote by $(\Gamma_X)_\infty$ the set of sinks of all the hyperbolic elements of $\Gamma_X$, which is easily seen to be $\Gamma_X$-invariant.
    Finally we recall that if $X$ is a complete, Nielsen-convex hyperbolic surface, then $(\Gamma_X)_\infty$ is a dense subset of $S^1$ (see \cite[Corollary 3.4.5]{hubbard_2006_teichmueller}).
    
    Finally, we recall the \emph{Douady-Earle extension}.
    This is a construction that extends homeomorphisms of $S^1$ to homeomorphisms of $\HH^2$ in a conformally natural way.
    That is, $\mathup{DE} : \Homeo^+ (S^1) \to \Homeo^+ (\HH^2)$ is a function such that $\mathup{DE} (\sigma_1 \circ f \circ \sigma_2) = \sigma_1 \circ \mathup{DE} (f) \circ \sigma_2$ for every $f \in \Homeo^+ (S^1)$ and every $\sigma_1, \sigma_2 \in \mathup{PSL} (2, \RR)$.
    Further, $\mathup{DE}$ is continuous when $\Homeo^+ (\HH^2)$ is given the compact-open topology (see \cite[Proposition 2, Section 4]{douady_earle_1986_conformally_natural_extension}).
    
    \section{Key tool: The Homeomorphism at Infinity}
    \label{sec:homeomorphism_at_infinity}
    
    In this section, we provide a detailed proof of Proposition \ref{prop:homeomorphism_at_infinity}, which is actually a special case of \cite[Proposition 5.3]{thurston_1986_earthquakes}.
    In that paper, Thurston used it to reduce the earthquake theorem for general hyperbolic surfaces to a version of the earthquake theorem for their universal covers (that is, the hyperbolic plane).
    
    \begin{prop}[Homeomorphism at infinity] \label{prop:homeomorphism_at_infinity}
        Let $X, Y$ be complete hyperbolic surfaces. Fix universal covers $p_X : \HH^2 \to X$, $p_Y : \HH^2 \to Y$, and let $\Gamma_X, \Gamma_Y$ be the respective deck groups.
        Let $f : X \to Y$ be a homeomorphism with a lift $\tilde f : \HH^2 \to \HH^2$ to the universal covers.
        Assume that both $X$ and $Y$ are Nielsen-convex.
        Then
        \begin{enumerate}
            \item \label{prop:homeomorphism_at_infinity_extends} $\tilde f$ extends to a homeomorphism at infinity $\partial \tilde f : S^1 \to S^1$.
            \item \label{prop:homeomorphism_at_infinity_equivariant} $\partial \tilde f$ is $\Gamma_X$-equivariant.
            That is, for every $\gamma \in \Gamma_X$, we have $(\partial \tilde f) \circ \gamma = f_* (\gamma) \circ (\partial \tilde f)$.
            \item \label{prop:homeomorphism_at_infinity_isotopy} Further, if $f_t$ is an isotopy which lifts to an isotopy $\tilde f_t$, then $\partial \tilde f_0 = \partial \tilde f_1$.
            \item \label{prop:homeomorphism_at_infinity_composition} Suppose $Z$ is another complete, Nielsen-convex hyperbolic surface with a universal cover $p_Z : \HH^2 \to Z$. If $g : Y \to Z$ is another homeomorphism with lift $\tilde g$ then $\partial (\tilde g \circ \tilde f) = (\partial \tilde g) \circ (\partial \tilde f)$.
            \item \label{prop:homeomorphism_at_infinity_isometrycomp} If $\sigma \in \mathup{PSL}(2, \RR)$, then $\sigma \circ \tilde f$ also extends to a homeomorphism at infinity, and $\partial (\sigma \circ \tilde f) = \sigma \circ (\partial \tilde f)$.
        \end{enumerate}
    \end{prop}
    
    \begin{remark}
        Note that if $X, Y$ are closed surfaces, then $\tilde f$ is a quasiisometry and so extends to a quasisymmetric homeomorphism of $S^1$.
        However, for infinite type surfaces $X, Y$, $\tilde f$ may not be a quasiisometry.
        The proposition asserts that it still extends to $S^1$.
    \end{remark}
    \begin{remark}
        Note that the homeomorphism $\partial \tilde f$ depends on the choice of the lift $\tilde f$.
        The other choices of the lift of $f$ are $\gamma \circ \tilde f$ for $\gamma \in \Gamma_Y$ or $\tilde f \circ \gamma$ for $\gamma \in \Gamma_X$, leading to the homeomorphism at infinity $\gamma \circ (\partial \tilde f)$ or $(\partial \tilde f) \circ \gamma$ by part (\ref{prop:homeomorphism_at_infinity_isometrycomp}) of the proposition.
    \end{remark}

    \begin{proof}[Proof of Proposition \ref{prop:homeomorphism_at_infinity}]
        The outline of the proof is as follows:
        Let $f_* : \Gamma_X \to \Gamma_Y$ be the isomorphism of deck groups induced by $\tilde f$.
        First, in Lemma \ref{lem:type_preserving}, we prove that $f_*$ preserves the type (hyperbolic or parabolic) of elements of the deck group.
        Next, we define $\partial \tilde f$ on the set $(\Gamma_X)_\infty$ of sinks of hyperbolic elements of $\Gamma_X$.
        Recall that $(\Gamma_X)_\infty$ is a dense subset of $S^1$.
        Then, in Lemma \ref{lem:circular_order_preserving}, we show that $\partial \tilde f$ is monotonic, that is, it preserves that circular order of points in $S^1$.
        Finally we show that $\partial \tilde f$ has no jump discontinuities, allowing us to extend it to a unique homeomorphism of $S^1$.
        
        \begin{lemma}\label{lem:type_preserving}
            $f_*$ is a type preserving isomorphism.
        \end{lemma}
        \begin{proof}
            Since powers of hyperbolic and parabolic isometries of $\HH^2$ are hyperbolic and parabolic respectively, it is enough to prove that $f_*$ preserves the type of primitive elements.
            Suppose, for the sake of contradiction, that there is a primitive element $\gamma \in \Gamma_X$ such that $\gamma$ and $f_* (\gamma) \in \Gamma_Y$ are not of the same type.
            As the torsion-free discrete groups $\Gamma_X$ and $\Gamma_Y$ cannot have elliptics, one of $\gamma$ and $f_* (\gamma)$ is hyperbolic and the other is parabolic.
            Replacing $f$ with $f^{-1}$ if necessary, assume that $\gamma$ is hyperbolic and $f_* (\gamma)$ is parabolic.
            Since $f_*$ is an isomorphism, $f_* (\gamma)$ is a primitive parabolic element in the group $\Gamma_Y$.
            Thus there is a horodisk in $\HH^2$ about its fixed point which projects, under $p_Y$, to a cusp neighbourhood $U$ of an end $e$ of $Y$.
            \begin{figure}\label{fig:type_preserving_isomorphism}
                \centering
                \def\svgwidth{\textwidth}
                \resizebox{0.6\textwidth}{!}{\Large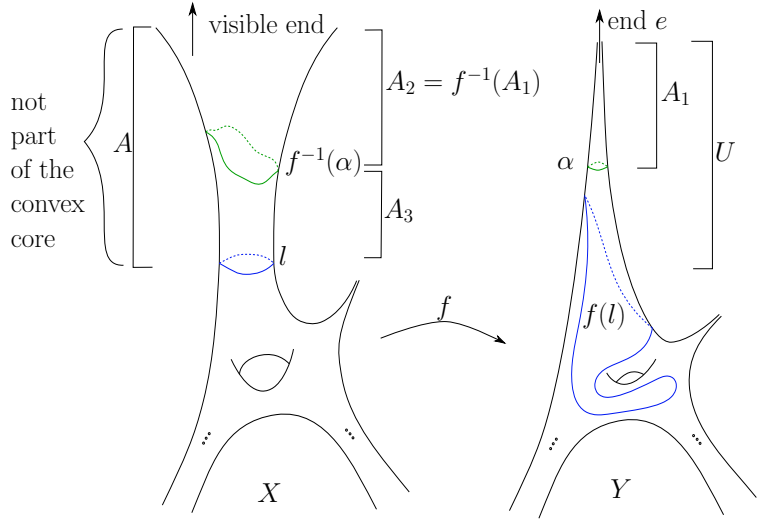}
                \caption{$f_*$ is a type preserving isomorphism}
            \end{figure}
            See Figure 2. 
            Let $L \subset \HH^2$ be the axis of the hyperbolic element $\gamma$ and $l \subset X$ be its projection under $p_X$.
            Then $l$ is a closed geodesic of $X$ and hence compact.
            Thus its image $f (l)$ is also compact.
            Hence there exists a horocyclic oriented closed curve $\alpha \subset U \subset Y$, in a sufficiently small neighbourhood of the end $e$, that is disjoint from $f (l)$, and such that the holonomy around $\alpha$ is $f_* (\gamma)$.
            Since $f_* (\gamma)$ is a primitive element, $\alpha$ is a simple closed curve.
            Further, $\alpha$ cuts $Y$ into two components, one of which is (topologically) an annulus $A_1 \subset U$, and the other contains $f (l)$.
            Since $f$ is a homeomorphism, $f^{-1} (\alpha)$ cuts $X$ into two components, one of which is the annulus $A_2 = f^{-1} (A_1)$, and the other contains $l$.
            
            Now $l$ and $f^{-1} (\alpha)$ are closed curves in the same free homotopy class, since the holonomy around both is $\gamma$, and $l$ is the geodesic representative in this class.
            Since $f^{-1} (\alpha)$ is a simple curve, so is $l$.
            Further, $l$ and $f^{-1} (\alpha)$ are disjoint because $f (l)$ and $\alpha$ are disjoint.
            $l$ and $f^{-1} (\alpha)$ are disjoint and isotopic, and hence bound an annulus $A_3 \subset X \setminus A_2$.
            Thus $l$ cuts $X$ into two components, one of which is $A = A_2 \cup f^{-1} (\alpha) \cup A_3$, an annulus.
            The closure of $A$ is a complete hyperbolic surface with geodesic boundary $l$.
            Topologically it is an annulus with one boundary component, and thus is a hyperbolic funnel.
            That is, it is the quotient of the half plane in $\HH^2$ bounded by $L$ by the isometry $\gamma$ translating along $L$.
            But this means that $X$ has a visible end $f_*^{-1} (e)$.
            Here $f_*$ is the induced map between spaces of ends induced by the map $f$.
            In other words, $A$ is not part of the convex core $C(X)$ of $X$.
            This contradicts our assumption that $X$ is a complete, Nielsen-convex hyperbolic surface.
            We conclude that $f_*$, the induced map between deck groups, is type preserving.
        \end{proof}
        
        Now we define $\partial \tilde f$ on a dense subset of $S^1$.
        Let $(\Gamma_X)_\infty$ and $(\Gamma_Y)_\infty$ be the sets of sinks of all the hyperbolic elements of $\Gamma_X$ and $\Gamma_Y$ respectively.
        Suppose $\gamma \in \Gamma_X$ is a hyperbolic element.
        By Lemma \ref{lem:type_preserving}, $f_*$ is type preserving so $f_* (\gamma)$ is also a hyperbolic element.
        \begin{definition}
            We define $\partial \tilde f (\gamma_\infty)$ to be the sink $(f_* (\gamma))_\infty$.
        \end{definition}
        This is well defined, because if $\gamma_1, \gamma_2$ are hyperbolic elements in the discrete group $\Gamma_X$ that have the same sink, then they are in fact positive powers of a hyperbolic element $\gamma_3 \in \Gamma_X$.
        Thus $f_* (\gamma_1)$ and $f_* (\gamma_2)$ are positive powers of $f_* (\gamma_3)$ and hence all three have the same sink.
        Similar considerations for $\tilde f^{-1}$ show that the function $\partial \tilde f$ so defined is a bijection from $(\Gamma_X)_\infty$ to $(\Gamma_Y)_\infty$.
        
        There is another way of viewing the function $\partial \tilde f$ as follows.
        \begin{lemma} \label{lem:extension_alternate}
            Let $L$ be an oriented geodesic line in $\HH^2$ whose forward endpoint at infinity is the sink $\gamma_\infty$ of a hyperbolic element $\gamma \in \Gamma_\infty$.
            Then the oriented curve $\tilde f (L)$ has forward endpoint $(f_*(\gamma))_\infty$.
        \end{lemma}
        \begin{proof}
            First consider the case where $L$ is the axis of $\gamma$.
            Then $\gamma$ acts by translation along $L$.
            Hence $f_* (\gamma)$ acts by translation along $\tilde f (L)$.
            Therefore $\tilde f (L)$ joins the source of $\gamma$ to the sink of $\gamma$.
            That is, $(f_* (\gamma))_\infty$ is the forward endpoint of $\tilde f (L)$.
            Next, let $L'$ be any oriented geodesic line whose forward endpoint is $\gamma_\infty$.
            We will show that the forward endpoint of $\tilde f (L')$ is $(f_* (\gamma))_\infty$.
            
            Let $\tilde K$ denote the $1$-neighbourhood of $L$ in $\HH^2$, that is, the set of points of distance at most $1$ from $L$.
            Let $K = p_X (\tilde K)$.
            The set $K$ is compact, and so is its image $f (K) \subset Y$.
            Since injectivity radius is a continuous function on any Riemannian manifold, the injectivity radius of points in $f (K)$ is bounded below by a number $\varepsilon > 0$.
            Also since $K$ is compact, $f|_K$ is uniformly continuous.
            Let $\delta > 0$ be such that for all $x, y \in K$, $d_X (x, y) < \delta$ implies $d_Y (f (x), f (y)) < \varepsilon$.
            Reducing $\delta$ if necessary, assume that $\delta \le 1$.
            
            \begin{figure}\label{fig:homeomorphism_at_infinity}
                \centering
                \def\svgwidth{\textwidth}
                \resizebox{\textwidth}{!}{\footnotesize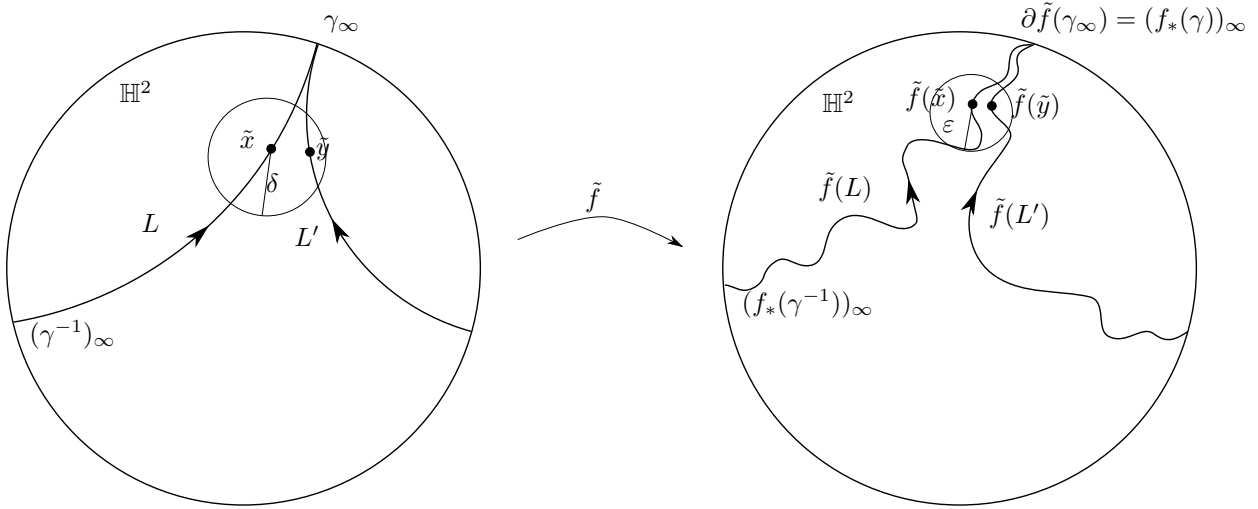}
                \caption{$\partial \tilde f$ for the sinks $\gamma_\infty$.}
            \end{figure}
            Since $L$ and $L'$ have the same forward endpoint, they are asymptotic.
            Let $L'_1$ be the subray of $L'$ that is within the $\delta$-neighbourhood of $L$.
            For any point $\tilde y \in L'_1$, there exists a point $\tilde x \in L$ such that $d_{\HH^2} (\tilde x, \tilde y) < \delta$.
            Since $\delta \le 1$, we have $d_{\HH^2} (\tilde x, \tilde y) < 1$ and so $\tilde x, \tilde y \in \tilde K$.
            Let $\alpha$ be the geodesic segment joining $\tilde x$ to $\tilde y$, and let $x = p_X (\tilde x)$, $y = p_X (\tilde y)$.
            See Figure 3. 
            Since $p_X$ is a local isometry, $d_X (x, y) \le d_{\HH^2} (\tilde x, \tilde y) < \delta$.
            Hence by uniform continuity as above, $d_Y (f (x), f (y)) < \varepsilon$.
            In fact, the diameter of $p_X (\alpha)$ is less than $\delta$ so $f (p_X (\alpha))$ is a curve that lies entirely in the $\varepsilon$-neighbourhood of the point $f (x)$.
            Since the injectivity radius at $f (x)$ is at least $\varepsilon$, $p_Y$ is an isometry between the balls of radius $\varepsilon$ centred at the points $\tilde f (\tilde x)$ in $\HH^2$ and its $p_Y$-image $f (x)$ in $Y$.
            Hence $f (p_X (\alpha))$ lifts to a curve contained in the ball of radius $\varepsilon$ centred at the point $\tilde f (\tilde x)$ and whose endpoint is $\tilde f (\tilde y)$.
            That is to say, $\tilde f (\tilde y)$ is contained in the ball of radius $\varepsilon$ centred at the point $\tilde f (\tilde x)$.
            In other words, $d_{\HH^2} (\tilde f (\tilde x), \tilde f (\tilde y)) < \varepsilon$.
            But $\tilde f (\tilde y)$ is an arbitrary point on $\tilde f (L'_1)$ and $\tilde f (\tilde x)$ lies on $\tilde f (L)$, so we conclude that the curve $\tilde f (L'_1)$ is at a bounded distance from $\tilde f (L)$.
            Therefore $\tilde f (L'_1)$ limits to a point on the boundary $S^1$ and further forward endpoints of $\tilde f (L)$ and $\tilde f (L'_1)$ are the same, namely $(f_* (\gamma))_\infty$.
            Thus $\tilde f (L')$ also has forward endpoint $(f_* (\gamma))_\infty$.
        \end{proof}
        
        \begin{lemma}\label{lem:circular_order_preserving}
            $\partial \tilde f$ (so far defined on the dense subset $(\Gamma_X)_\infty$) is monotonic, that is, it preserves the circular order on $S^1$.
            In particular, if $a, b, c \in (\Gamma_X)_\infty$ and $(a, b, c)$ is a positively oriented triple, then so is the triple $(\partial \tilde f (a), \partial \tilde f (b), \partial \tilde f(c))$.
        \end{lemma}
        \begin{proof}
            Suppose $a, b, c \in (\Gamma_X)_\infty$ such that $(a, b, c)$ is a positively oriented triple.
            Let $L_1, L_2, L_3$ be oriented geodesic lines joining $a$ to $b$, $b$ to $c$ and $c$ to $a$ respectively, and let $T$ be the ideal geodesic triangle with vertices $a, b, c$.
            Then by Lemma \ref{lem:extension_alternate}, $\tilde f$ maps $T$ to an ideal triangle $\tilde f (T)$ (whose sides $\tilde f (L_1), \tilde f (L_2), \tilde f (L_3)$ are not necessarily geodesics) with vertices at $\partial \tilde f (a), \partial \tilde f (b), \partial \tilde f (c)$.
            Now $L_1, L_2, L_3$ form the boundary of the triangle $T$, with its induced border orientation.
            Since $f$ is orientation preserving and hence so is $\tilde f$, we conclude that the orientation of $\tilde f (L_1) \cup \tilde f (L_2) \cup \tilde f (L_3)$ matches the orientation induced as the boundary of the triangle $\tilde f (T)$.
            Therefore the vertices $(\partial \tilde f (a), \partial \tilde f (b), \partial \tilde f (c))$ of the image triangle $\tilde f (T)$ form a positively oriented triple.
            This concludes the proof of monotonicity of $\partial \tilde f$.
        \end{proof}
        
        Since $\partial \tilde f$ is monotonic, it can be extended uniquely and continuously to the closure $\overline{(\Gamma_X)_\infty}$, unless there are jump discontinuities.
        However the image $\partial \tilde f ((\Gamma_X)_\infty) = (\Gamma_Y)_\infty$ is dense in $S^1$ since $Y$ is a complete, Nielsen-convex hyperbolic surface.
        Therefore there cannot be any jump discontinuities, so $\partial \tilde f$ can be extended uniquely and continuously to the closure $\overline{(\Gamma_X)_\infty}$.
        Since $X$ is also a complete, Nielsen-convex hyperbolic surface, we have $\overline{(\Gamma_X)_\infty} = S^1$, so the extension is defined over all of $S^1$.
        Further $\partial (\tilde f^{-1})$ is clearly a continuous inverse to $\partial \tilde f$, so $\partial \tilde f$ is a homeomorphism of $S^1$.
        
        Now we prove the remaining assertions of the proposition.
        $\partial \tilde f$ is $\Gamma_X$ equivariant because $\tilde f$ is too.
        If $f_t$ is an isotopy which lifts to an isotopy $\tilde f_t$, then we have $f_{0*} = f_{1*}$, and hence for every hyperbolic $\gamma \in \Gamma_X$, $\partial \tilde f_0 (\gamma_\infty) = f_{0*} (\gamma)_\infty = f_{1*} (\gamma)_\infty = \partial \tilde f_1 (\gamma_\infty)$.
        Therefore the continuous functions $\partial \tilde f_0$ and $\partial \tilde f_1$ agree on the dense set $(\Gamma_X)_\infty$, and hence are equal.
        Next, we compute that for every hyperbolic $\gamma \in \Gamma_X$, $\partial (\tilde g \circ \tilde f) (\gamma_\infty) = ((g \circ f)_* (\gamma))_\infty = (g_* (f_* (\gamma)))_\infty = \partial \tilde g ((f_* (\gamma))_\infty) = \partial \tilde g \circ \partial \tilde f (\gamma_\infty)$.
        Again the continuous functions $\partial (\tilde g \circ \tilde f)$ and $(\partial \tilde g) \circ (\partial \tilde f)$ agree on the dense set $(\Gamma_X)_\infty$, and hence are equal.
        For the final assertion in the proposition, we note that $\sigma \in \mathup{PSL} (2, \RR)$ also extends to a homeomorphism at infinity, which recall that we are denoting by the same symbol $\sigma$.
        Therefore $\sigma \circ \tilde f$ extends to a homeomorphism at infinity, and $\partial (\sigma \circ \tilde f) = \sigma \circ (\partial \tilde f)$.
    \end{proof}
    
    \section{The Topology of the Marked Moduli Space}
    \label{sec:topology}
    
    In this section, we define the topology on the set $\mathcal T (S)$ (Definition \ref{def:marked_moduli_space_topology}).
    Fix a universal cover $p : \HH^2 \to S$ with deck group $\Gamma \subset \mathup{PSL} (2, \RR)$, which is a torsion-free Fuchsian group of the first kind, that is, the limit set of $\Gamma$ is $S^1$.
    We define the topology on $\mathcal T (S)$ via a bijection $\Phi_p$ onto the topological space $\mathcal T (p)$ defined below:
    \begin{definition}[Alternate description of the marked moduli space]
        \begin{equation} \label{eq:alt_marked_moduli_space}
        \begin{split}
            \mathcal T (p) &= \leftQ{\widetilde {\mathcal T} (p)}{\mathup{PSL} (2, \RR)} \text{ where }\\
            \widetilde {\mathcal T} (p) &= \left\{F \in \Homeo^+ (S^1) \middle \vert F \circ \Gamma \circ F^{-1} \subset \mathup{PSL} (2, \RR)\right\}
        \end{split}
        \end{equation}
    \end{definition}
    $\mathcal T (p)$ is naturally a topological space as follows.
    $\widetilde {\mathcal T} (p)$ inherits a topology as a subspace of $\Homeo^+ (S^1)$ with the compact-open topology.
    $\mathcal T (p)$ has the quotient topology, where $\mathup{PSL} (2, \RR)$ acts on $\widetilde {\mathcal T} (p)$ by multiplication on the left.
    We denote by $\pi_p$ the quotient map $\widetilde {\mathcal T} (p) \to \mathcal T (p)$.
    
    Now we define $\Phi_p$.
    Suppose $[X, f]$ is a marked hyperbolic structure.
    Choose a universal cover $p_X : \HH^2 \to X$.
    The marking map $f : S \to X$ lifts to a homeomorphism $\tilde f : \HH^2 \to \HH^2$, which, by Proposition \ref{prop:homeomorphism_at_infinity} (\ref{prop:homeomorphism_at_infinity_extends}), extends to a homeomorphism $F = \partial \tilde f : S^1 \to S^1$.
    \begin{definition}\label{def:Phi_p}
        We define $\Phi_p [X, f]$ to be the right coset $[F] = [\partial \tilde f] = \mathup{PSL} (2, \RR) \circ (\partial \tilde f)$ of $\mathup{PSL} (2, \RR)$ in $\Homeo^+ (S^1)$.
    \end{definition}
    
    \begin{prop}
        $\Phi_p$ is a well defined function, and maps $\mathcal T (S)$ into $\mathcal T (p)$.
    \end{prop}
    \begin{proof}
        \begin{figure} \label{fig:Phi_p_independent}
            \centering
            \begin{subfigure}[b]{0.4\textwidth} \label{fig:Phi_p_well_defined_cover}
                \centering
                $\begin{tikzcd}[cramped]
                    \HH^2 \arrow[r,"\tilde f"] \arrow[d,"p"'] & \HH^2 \arrow[d, "p_X"'] \arrow[r,"\sigma"] & \HH^2 \arrow[ld, "q_X"] \\
                    S \arrow[r,"f"] & X &
                \end{tikzcd}$
                \caption{$\Phi_p [X, f]$ is independent of the cover $p_X~:~\HH^2~\to~X$.}
            \end{subfigure}
            \hfill
            \begin{subfigure}[b]{0.4\textwidth} \label{fig:Phi_p_well_defined_representative}
                \centering
                $\begin{tikzcd}[cramped]
                    \HH^2 \arrow[r,"\tilde f"'] \arrow[d,"p"'] \arrow[rr,"\tilde g"',bend left=50] & \HH^2 \arrow[d, "p_X"'] \arrow[r,"\sigma"'] & \HH^2 \arrow[d, "p_Y"'] \\
                    S \arrow[r,"f"'] \arrow[rr,"g",bend right=50] & X \arrow[r,"\varphi"'] & Y
                \end{tikzcd}$
                \caption{$\Phi_p [X, f]$ is independent of the representative $(X, f)$.}
            \end{subfigure}
            \caption{$\Phi_p [X, f]$ is well defined}
        \end{figure}
        We need to show that $\Phi_p [X, f]$ is independent of the chosen lift $\tilde f$ of $f$, independent of the chosen universal cover $p_X : \HH^2 \to X$, and independent of the chosen representative $(X, f)$ of the marked hyperbolic structure $[X, f]$.
        Indeed, any other lift of $f$ is of the form $\sigma \circ \tilde f$, where $\sigma \in \Gamma_X$ is a deck transformation.
        This extends at infinity to $\sigma \circ (\partial \tilde f)$, by Proposition \ref{prop:homeomorphism_at_infinity} (\ref{prop:homeomorphism_at_infinity_isometrycomp}).
        Since $\sigma \in \mathup{PSL} (2, \RR)$, we have an equality of right cosets $\mathup{PSL} (2, \RR) \circ (\sigma \circ \partial \tilde f) = \mathup{PSL} (2, \RR) \circ (\partial \tilde f)$, that is, $[\sigma \circ (\partial \tilde f)] = [\partial \tilde f]$.
        
        Next, suppose that $q_X : \HH^2 \to X$ is another universal cover.
        Since universal covers are unique up to isometry, there is a $\sigma \in \mathup{PSL} (2, \RR)$ such that $p_X = q_X \circ \sigma$.
        See Figure 4(A)
        , which commutes. 
        Then $\sigma \circ \tilde f$ is a lift of $f$ with respect to the universal covers $p$ and $q_X$, which extends at infinity to $\sigma \circ (\partial \tilde f)$, by \ref{prop:homeomorphism_at_infinity} (\ref{prop:homeomorphism_at_infinity_isometrycomp}).
        Again we have $[\sigma \circ (\partial \tilde f)] = [\partial \tilde f]$.
        
        Next, suppose $(Y, g)$ is another representative of the same marked hyperbolic structure.
        Then there is an isometry $\varphi : X \to Y$ such that $\varphi \circ f$ is homotopic to $g$.
        Then $\varphi$ lifts to an isometry $\sigma \in \mathup{PSL} (2, \RR)$.
        For a lift $\tilde f$ of $f$, $\sigma \circ \tilde f$ is a lift of $\varphi \circ f$.
        The homotopy from $\varphi \circ f$ to $g$ lifts to a homotopy from $\sigma \circ f$ to a lift $\tilde g$ of $g$.
        See Figure 4(B)
        , in which the squares commute and the top and bottom triangles homotopy commute.
        By Proposition \ref{prop:homeomorphism_at_infinity} (\ref{prop:homeomorphism_at_infinity_isometrycomp}), we have $\partial \tilde g = \partial (\sigma \circ \tilde f) = \sigma \circ (\partial \tilde f)$.
        Therefore again we have $[\partial \tilde g] = [\sigma \circ (\partial \tilde f)] = [\partial \tilde f]$.
        Hence $\Phi_p[X, f]$ is well defined.
        
        Finally, recall from Proposition \ref{prop:homeomorphism_at_infinity} (\ref{prop:homeomorphism_at_infinity_equivariant}) that $\tilde f$ is $\Gamma$-equivariant, that is, for every $\gamma \in \Gamma$ we have $\tilde f \circ \gamma = f_* (\gamma) \circ \tilde f$.
        Thus $\tilde f \circ \Gamma = \Gamma_X \circ \tilde f$, or in other words, $\Gamma_X = \tilde f \circ \Gamma \circ \tilde f^{-1}$.
        Taking the induced maps at infinity, by Proposition \ref{prop:homeomorphism_at_infinity} (\ref{prop:homeomorphism_at_infinity_equivariant}), we have $\Gamma_X = (\partial \tilde f) \circ \Gamma \circ (\partial \tilde f)^{-1} \subset \mathup{PSL} (2, \RR)$.
        So $\Phi_p$ maps $\mathcal T (S)$ into $\mathcal T (p)$ indeed.
    \end{proof}
    
    \begin{prop}
        \label{prop:Phi_p_bijection}
        $\Phi_p : \mathcal T (S) \to \mathcal T (p)$ is a bijection.
    \end{prop}
    \begin{proof}
        To see that $\Phi_p$ is injective, suppose $\Phi_p [X, f] = \Phi_p [Y, g]$.
        Let $\tilde f, \tilde g$ denote the lifts of $f, g$ with respect to universal covers $p_X : \HH^2 \to X$ and $p_Y : \HH^2 \to Y$.
        Note that $(\partial \tilde f) \circ \Gamma \circ (\partial \tilde f)^{-1} = \Gamma_X$ and $(\partial \tilde g) \circ \Gamma \circ (\partial \tilde g)^{-1} = \Gamma_Y$.
        Also, $[\partial \tilde g] = [\partial \tilde f]$, so there exists a $\sigma \in \mathup{PSL} (2, \RR)$ such that $\partial \tilde g = \sigma \circ (\partial \tilde f)$.
        We have $\sigma \circ \Gamma_X = \sigma \circ (\partial \tilde f) \circ \Gamma \circ (\partial \tilde f)^{-1} = (\partial \tilde g) \circ \Gamma \circ (\partial \tilde f)^{-1} = \Gamma_Y \circ (\partial \tilde g) \circ (\partial \tilde f)^{-1} = \Gamma_Y \circ \sigma$.
        Thus $\sigma$ is equivariant, hence descends to an isometry $\varphi : X \to Y$.
        Further, the homotopy from $\sigma \circ \tilde f$ to $\tilde g$ that moves each point along a geodesic line at constant speed is also equivariant, hence descends to a homotopy from $\varphi \circ f$ to $g$.
        Therefore $[X, f] = [Y, g]$ in $\mathcal T (S)$, so $\Phi_p$ is injective.
        
        To see that $\Phi_p$ is surjective, suppose $[F] \in \mathcal T (p)$.
        Then define $\Gamma_X = F \circ \Gamma \circ F^{-1}$.
        Since $F \in \widetilde {\mathcal T} (p)$, we have $\Gamma_X \subset \mathup{PSL} (2, \RR)$.
        $\Gamma_X$ is discrete because conjugation by $F$ is a homeomorphism of $\Homeo^+ (S^1)$, and that $\Gamma$ is discrete.
        $\Gamma_X$ is also torsion-free because it is isomorphic to $\Gamma$, which is torsion-free.
        Further $F$ maps the limit set of $\Gamma$ to the limit set of $\Gamma_X$, which is therefore all of $S^1$.
        Hence $X = \leftQ{\HH^2}{\Gamma_X}$ is a complete, Nielsen-convex hyperbolic surface.
        The Douady-Earle extension of $F$ is $\Gamma$-equivariant, hence descends to a homeomorphism $f : S \to X$.
        Clearly $\Phi_p[X, f] = [F]$, so $\Phi_p$ is surjective.
        This completes the proof of bijectivity of $\Phi_p$.
    \end{proof}
    
    The above proposition allows us to define a topology on $\mathcal T (S)$.
    \begin{definition}[Topology of the Marked Moduli Space]
        \label{def:marked_moduli_space_topology}
        The topology on $\mathcal T (S)$ is defined by declaring the bijection $\Phi_p : \mathcal T (S) \to \mathcal T (p)$ to be a homeomorphism.
        In other words, a subset $U \subset \mathcal T (S)$ is open if and only if $\Phi_p (U) \subset \mathcal T (p)$ is open.
    \end{definition}
    
    \begin{prop}
        If $p' : \HH^2 \to S$ is another universal cover, then $\Phi_{p'} \circ \Phi_p^{-1} : \mathcal T (p) \to \mathcal T (p')$ is a homeomorphism.
        Thus the topology on $\mathcal T (S)$ does not depend on the choice of the universal cover $p$.
    \end{prop}
    \begin{proof}
        \begin{figure} \label{fig:topology_well_defined}
            \centering
            \begin{subfigure}[b]{0.4\textwidth} \label{fig:topology_lift}
                \centering
                $\begin{tikzcd}[cramped]
                    \HH^2 \arrow[r,"\tilde f"'] \arrow[d,"p"'] & \HH^2 \arrow[d, "p_X"'] & \HH^2 \arrow[l,"\tilde f'"] \arrow[d, "p'"'] \arrow[ll,"h",bend right=50] \\
                    S \arrow[r,"f"'] & X & S \arrow[l,"f"] \arrow[ll,"\operatorname{id}_S"',bend left=50]
                \end{tikzcd}$
                \caption{The identity lifts to the map $h$ with respect to covers $p'$ and $p$.}
            \end{subfigure}
            \hfill
            \begin{subfigure}[b]{0.4\textwidth} \label{fig:topology_homeo}
                \centering
                $\begin{tikzcd}[cramped]
                    \Homeo^+ (S^1) \arrow[r,"\circ \partial h","\cong"'] & \Homeo^+ (S^1) \\
                    \widetilde{\mathcal T} (p) \arrow[u,hook] \arrow[r,"\circ \partial h","\cong"'] \arrow[d,two heads,"\pi_p"'] & \widetilde{\mathcal T} (p') \arrow[u,hook] \arrow[d,two heads,"\pi_{p'}"'] \\
                    \mathcal T (p) \arrow[r,"\Phi_{p'} \circ \Phi_p^{-1}"'] & \mathcal T (p')
                \end{tikzcd}$
                \caption{$\Phi_{p'} \circ \Phi_p^{-1}$ is multiplication by $\partial h$ on the right.}
            \end{subfigure}
            \caption{$\Phi_{p'} \circ \Phi_p^{-1}$ is a homeomorphism.}
        \end{figure}
        Suppose $p, p' : \HH^2 \to S$ are two universal covers both of whose deck groups are torsion-free Fuchsian groups of the first kind.
        Then the identity map of $S$ lifts, with respect to the covers $p'$ and $p$, to a a map $h : \HH^2 \to \HH^2$ such that $p \circ h = p'$.
        By Proposition \ref{prop:homeomorphism_at_infinity} (\ref{prop:homeomorphism_at_infinity_extends}), $h$ extends to a homeomorphism $\partial h$ at infinity.
        Note that $h$ is not necessarily an isometry, because the two covers $p$ and $p'$ can possibly induce different hyperbolic metrics on $S$.
        
        Let $[X, f] \in \mathcal T (S)$ be a marked hyperbolic structure and $p_X : \HH^2 \to X$ be a universal cover.
        The marking map $f$ lifts to $\tilde f$ with respect to the covers $p$ and $p_X$, so that $\Phi_p [X, f] = [\partial \tilde f]$.
        That is, the right $\mathup{PSL} (2, \RR)$-coset $\Phi_p [X, f]$ is represented by the homeomorphism $\partial \tilde f$.
        Then the map $\tilde f'$ defined by $\tilde f' = \tilde f \circ h$ is a lift of $f$ with respect to the covers $p'$ and $p_X$.
        See Figure 5(A), which commutes. 
        Thus $\Phi_{p'} [X, f]$ is represented by $\partial (\tilde f \circ h) = (\partial \tilde f) \circ (\partial h)$, where the equality is due to Proposition \ref{prop:homeomorphism_at_infinity} (\ref{prop:homeomorphism_at_infinity_composition}).
        Hence $\Phi_{p'} \circ \Phi_p^{-1} : \mathcal T (p) \to \mathcal T (p')$ is induced simply by multiplication on the right by $\partial h$.
        Multiplication on the right by the fixed map $\partial h$ is a homeomorphism of the topological group $\Homeo^+ (S^1)$, and it restricts to a homeomorphism from $\widetilde{\mathcal T} (p)$ to $\widetilde{\mathcal T} (p')$, which descends to a homeomorphism $\Phi_{p'} \circ \Phi_p^{-1}$ of the spaces of right cosets $\mathcal T (p)$ and $\mathcal T (p')$.
        See Figure 5(B). 
        
        Now we can show that the topology of $\mathcal T (S)$ does not depend on the universal cover $p$.
        Indeed, for another universal cover $p'$ and any subset $U \subset \mathcal T (S)$, the condition $\Phi_p (U)$ is open in $\mathcal T (p)$ is equivalent, since $\Phi_{p'} \circ\Phi_p^{-1}$ is a homeomorphism, to the condition that $\Phi_{p'} (U) = \Phi_{p'} \circ\Phi_p^{-1} (\Phi_p (U))$ is open in $\mathcal T (p')$.
    \end{proof}
    
    \section{Continuity of the action}
    \label{sec:continuity}
    
    Now that we have a well defined topology on $\mathcal T (S)$, we can ask if the action of $\MCG (S)$ on $\mathcal T (S)$ is continuous.
    In this section, we prove that it is indeed continuous.
    It will follow that for each mapping class $[\psi] \in \MCG (S)$, its action on $\mathcal T (S)$, that is, the function $A_{[\psi]} : \mathcal T (S) \to \mathcal T (S)$, is a homeomorphism.
    We can also ask if the action representation $A : \MCG (S) \to \Homeo (\mathcal T (S))$ is continuous.
    However, we have not defined any topology on the codomain.
    For infinite type surfaces $S$, we do not expect $\mathcal T (S)$ to be locally compact in general.
    So the compact-open topology may not be very useful.
    Instead we prove that the action function $A : \MCG (S) \times \mathcal T (S) \to \mathcal T (S)$ is continuous.
    Since the topology on $\mathcal T (S)$ is defined via homeomorphisms at infinity, we must compute the action function in terms of homeomorphisms at infinity in order to prove continuity.
    
    First we fix a universal cover $p : \HH^2 \to S$ with deck group $\Gamma \subset \mathup{PSL} (2, \RR)$, a torsion-free Fuchsian group of the first kind as before.
    Note that $\Gamma$ is also naturally a subgroup of $\Homeo^+ (S^1)$.
    Let $\mathup{N} (\Gamma)$ be the normaliser of $\Gamma$ in $\Homeo^+ (S^1)$ and let $\pi_N$ be the quotient map $\mathup N (\Gamma) \to \mathup N (\Gamma)/\Gamma$.
    We now define a function $\Psi_p : \MCG (S) \to \mathup {N} (\Gamma)/\Gamma$ which will enable us to compute the action function in terms of homeomorphisms at infinity.
    Consider a mapping class $[\psi] \in \MCG (S)$ where $\psi \in \Homeo^+ (S)$ is a representative homeomorphism.
    $\psi$ lifts to a homeomorphism $\tilde \psi : \HH^2 \to \HH^2$ which, by Proposition \ref{prop:homeomorphism_at_infinity} (\ref{prop:homeomorphism_at_infinity_extends}), extends to a homeomorphism at infinity $G = \partial \tilde \psi : S^1 \to S^1$.
    \begin{definition}
        We define $\Psi_p [\psi]$ to be the right coset $[G] = [\partial \tilde \psi] = \Gamma \circ (\partial \tilde \psi)$  of $\Gamma$ in $\Homeo^+ (S^1)$.
    \end{definition}
    \begin{prop}
        $\Psi_p$ is a well defined function, and a group isomorphism from $\MCG (S)$ to $\mathup {N} (\Gamma)/\Gamma$.
    \end{prop}
    \begin{proof}
        We need to show that $\Psi_p [\psi]$ is independent of the chosen representative homeomorphism $\psi$ and independent of the chosen lift $\tilde \psi$ of the mapping class $[\psi]$.
        Firstly, if $\psi'$ is another representative of the mapping class $[\psi]$, then $\psi$ is homotopic to $\psi'$.
        This homotopy lifts to a homotopy between $\tilde \psi$ and a lift $\tilde \psi'$ of $\psi'$.
        By Proposition \ref{prop:homeomorphism_at_infinity} (\ref{prop:homeomorphism_at_infinity_isotopy}), we have $\partial \tilde \psi' = \partial \tilde \psi$, so $[\partial \tilde \psi] = [\partial \tilde \psi]$.
        Secondly, the set of all lifts of $\psi$ is $\Gamma \circ \tilde \psi$.
        Thus any other lift of $\psi$ is of the form $\sigma \circ \tilde \psi$ for some $\sigma \in \Gamma$, which extends at infinity to $\sigma \circ (\partial \tilde \psi)$ by Proposition \ref{prop:homeomorphism_at_infinity} (\ref{prop:homeomorphism_at_infinity_isometrycomp}).
        Hence we have $[\sigma \circ (\partial \tilde \psi)] = [\partial \tilde \psi]$, and so $\Psi_p$ is a well defined function.
        
        Further, it follows from $\Gamma \circ \tilde \psi = \tilde \psi \circ \Gamma$ (as homeomorphisms of $\HH^2$) and Proposition \ref{prop:homeomorphism_at_infinity} (\ref{prop:homeomorphism_at_infinity_equivariant}) that $\Gamma \circ (\partial \tilde \psi) = (\partial \tilde \psi) \circ \Gamma$ (as homeomorphisms of $S^1$), so in fact $(\partial \tilde \psi) \circ \Gamma \circ (\partial \tilde \psi)^{-1} = \Gamma$, and hence $\partial \tilde \psi$ lies in  the normaliser of $\Gamma$.
        So $\Psi_p$ maps $\MCG (S)$ into $\mathup N (\Gamma)/\Gamma$ indeed.
        A similar argument as with $\Phi_p$ in Proposition \ref{prop:Phi_p_bijection} shows that $\Psi_p$ is a bijection.
        If $\psi_1, \psi_2 \in \Homeo^+ (S)$ with lifts $\tilde \psi_1, \tilde \psi_2$ then $\tilde \psi_1 \circ \tilde \psi_2$ is a lift of $\psi_1 \circ \psi_2$.
        Thus $\Psi_p ([\psi_1] \cdot [\psi_2]) = \Psi_p [\psi_1 \circ \psi_2] = [\partial (\tilde \psi_1 \circ \tilde \psi_2)]$ which equals $[(\partial \tilde \psi_1) \circ (\partial \tilde \psi_2)]$ by Proposition \ref{prop:homeomorphism_at_infinity} (\ref{prop:homeomorphism_at_infinity_composition}), and hence equals $[\partial \tilde \psi_1] \circ [\partial \tilde \psi_2] = \Psi_p [\psi_1] \cdot \Psi_p [\psi_2]$.
        Hence $\Psi_p$ is a group homomorphism and, since it is a bijection as well, is a group isomorphism.
    \end{proof}
    
    We are ready to compute the action function in terms of homeomorphisms at infinity.
    The general idea is that $\Phi_p (A ([\psi], [X, f])) = \Phi_p [X, f] \circ (\Psi_p [\psi])^{-1}$.
    However, $\Phi_p [X, f]$ and $\Psi_p [\psi]$ are right cosets of $\mathup{PSL} (2, \RR)$ and $\Gamma$ respectively in $\Homeo^+ (S^1)$ and $\mathup{N} (\Gamma)$ respectively.
    In order to avoid dealing with multiplication of cosets, we state the computation in the following manner:
    
    \begin{prop} \label{prop:action_in_coordinates}
        Suppose $[X, f] \in \mathcal T (S)$ and $[\psi] \in \MCG(S)$.
        If $\Phi_p [X, f] = [F]$ and $\Psi_p [\psi] = [G]$, then we have
        \begin{equation} \label{eq:action_in_coordinates}
            \Phi_p (A ([\psi], [X, f])) = [F \circ G^{-1}]
        \end{equation}
    \end{prop}
    \begin{proof}
        Let $\tilde f$ be a lift of the marking map $f$, and let $\tilde \psi$ be a lift of the homeomorphism $\psi$ of $S$.
        Then $\Phi_p [X, f] = [\partial \tilde f]$ and $\Psi_p [\psi] = [\partial \tilde \psi]$.
        We have $A ([\psi], [X, f]) = [X, f \circ \psi^{-1}]$.
        Then the modified marking $f \circ \psi^{-1}$ lifts to the map $\tilde f \circ \tilde \psi^{-1}$, and by Proposition \ref{prop:homeomorphism_at_infinity} (\ref{prop:homeomorphism_at_infinity_composition}) this extends at infinity to $\partial (\tilde f \circ \tilde \psi^{-1}) = (\partial \tilde f) \circ (\partial \tilde \psi)^{-1}$.
        Therefore $\Phi_p (A ([\psi], [X, f])) = \Phi_p [X, f \circ \psi^{-1}] = [(\partial \tilde f) \circ (\partial \tilde \psi)^{-1}]$.
        If $F$ is any other representative of the right coset $[\partial \tilde f]$ of $\mathup{PSL} (2, \RR)$ in $\Homeo^+ (S^1)$, then $F = \sigma \circ (\partial \tilde f)$ for some $\sigma \in \mathup{PSL} (2, \RR)$.
        Similarly if $G$ is any other representative of the (left or right) coset $[\partial \tilde \psi]$ of $\Gamma$ in $\mathup{N} (\Gamma)$, then $G = (\partial \tilde \psi) \circ \gamma^{-1}$ for some $\gamma \in \Gamma$.
        Then $F \circ G^{-1} = (\sigma \circ (\partial \tilde f)) \circ ((\partial \tilde \psi) \circ \gamma^{-1})^{-1} = \sigma \circ (\partial \tilde f) \circ \gamma \circ (\partial \tilde \psi)^{-1}$,
        which equals, by Proposition \ref{prop:homeomorphism_at_infinity} (\ref{prop:homeomorphism_at_infinity_equivariant}),
        $\sigma \circ f_* (\gamma) \circ (\partial \tilde f) \circ (\partial \tilde \psi)^{-1}$.
        Since $\sigma, f_* (\gamma) \in \mathup{PSL} (2, \RR)$, this represents the same right coset of $\mathup{PSL} (2, \RR)$ in $\Homeo^+ (S^1)$ as that of $(\partial \tilde f) \circ (\partial \tilde \psi)^{-1}$.
        In other words, $\Phi_p (A ([\psi], [X, f])) = [F \circ G^{-1}]$.
    \end{proof}
    
    Thus in terms of homeomorphisms at infinity, the action of the mapping class group is simply by multiplication on the right by the inverse.
    Since inversion and multiplication are continuous operations in the topological group $\Homeo^+ (S^1)$, the formula $(G, F) \mapsto F \circ G^{-1}$ defines a continuous function on $\mathup N (\Gamma) \times \widetilde {\mathcal T} (p) \subset \Homeo^+ (S^1) \times \Homeo^+ (S^1)$.
    This descends to the action function $A$, which is continuous as long as $[\partial \tilde \psi]$ depends continuously on $[\psi]$.
    That is, we need to show that $\Psi_p$ is continuous, which we now do in Lemma \ref{lem:mcg_normaliser}.
    
    \begin{lemma} \label{lem:mcg_normaliser}
        Fix a universal cover $p : \HH^2 \to S$ with deck group $\Gamma \subset \mathup{PSL} (2, \RR) \subset \Homeo^+ (S^1)$, which is a torsion-free Fuchsian group of the first kind as before.
        Then the function $\Psi_p : \MCG (S) \to \mathup{N} (\Gamma)/\Gamma$ is a homeomorphism.
    \end{lemma}
    \begin{proof}
        Fix a basepoint $\tilde s \in \HH^2$ and let $s = p (\tilde s)$.
        Let $d$ be the visual metric on $\partial \HH^2 = S^1$ induced by the basepoint $\tilde s$.
        Since we have fixed a universal cover $p : \HH^2 \to S$, we have access to the complete hyperbolic metric $d_S$ on $S$ and its injectivity radius function $\mathup{inj}_S$, which is a continuous function on $S$.
        
        For the continuity of $\Psi_p$, we only need to show that $\Psi_p$ is continuous at the identity $[\mathup{id}_S] \in \MCG (S)$, since $\Psi_p$ is a group isomorphism.
        Note that $\Psi_p [\mathup{id}_S] = [\partial \mathup{id}_{\HH^2}] = [\mathup{id}_{S^1}]$.
        Let $U$ be an open neighbourhood of $[\mathup{id}_{S^1}]$ in $\mathup{N} (\Gamma)/\Gamma$.
        We have to show that $\Psi_p^{-1} (U)$ is an open neighbourhood of $[\mathup{id}_S]$ in $\MCG (S)$.
        Since $\pi_N$ is continuous, the preimage $\pi_N^{-1} (U)$ is open in $\mathup N (\Gamma)$, and hence is the intersection with $\mathup N (\Gamma)$ of an open neighbourhood $U_1$ of $\mathup{id}_{S^1}$ in $\Homeo^+ (S^1)$.
        Therefore there is an $\varepsilon$ such that the ball in $\Homeo^+ (S^1)$ (in the metric of uniform convergence)
        \[
        B_d (\mathup{id}_{S^1}, \varepsilon) = \{\theta \in \Homeo^+ (S^1) \mid d (\theta (q), q) < \varepsilon \text{ for all } q \in S^1\}
        \]
        is contained in $U_1$.
        
        We will construct an open neighbourhood $V_1$ of $\mathup{id}_S$ in $\Homeo^+ (S)$ and for all $\psi \in V_1$, construct a lift $\tilde \psi$ such that $\partial \tilde \psi \in B_d (\mathup{id}_{S^1}, \varepsilon)$.
        Since $B_d (\mathup{id}_{S^1}, \varepsilon) \subset U_1$, this means that $\Psi_p [\psi] = [\partial \tilde \psi] \in \pi_N (U_1 \cap \mathup N (\Gamma)) = U$.
        Since quotienting by the group action of $\Homeo_0 (S)$ is an open map, the projection $V$ of $V_1$ in $\MCG (S)$ is an open set .
        It follows that $V$ is an open neighbourhood of $[\mathup{id}_S]$ and is contained in $\Psi_p^{-1} (U)$.
        It will follow that $\Psi_p$ is continuous at $[\mathup{id}_S]$.
        
        \begin{figure} \label{fig:continuity_sinks}
            \def\svgwidth{\textwidth}
            \resizebox{0.5\textwidth}{!}{\Large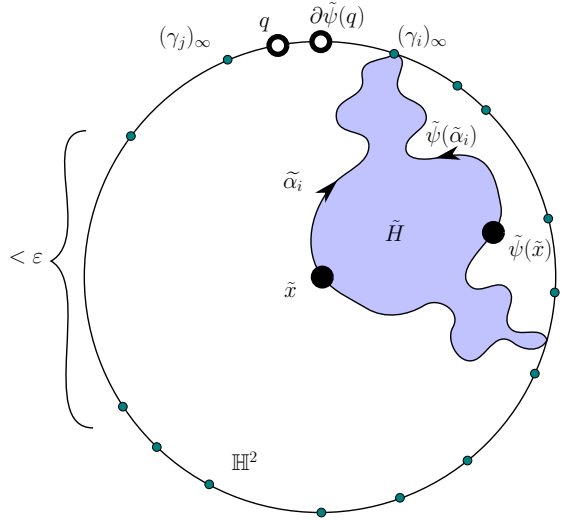}
            \caption{Lifts of the homeomorphisms in $V_1$ and their extensions at infinity.}
        \end{figure}
        
        To obtain the open neighbourhood $V_1$, we choose a finite set of hyperbolic elements $\gamma_1, \gamma_2, \ldots, \gamma_n \in \Gamma$, $n \ge 3$, such that the set $\{(\gamma_1)_\infty, (\gamma_2)_\infty, \ldots, (\gamma_n)_\infty\}$ of their sinks is $\frac{\varepsilon}{2}$-dense in $S^1$.
        This is possible because the set $(\Gamma)_\infty$ of all sinks of hyperbolic elements in $\Gamma$ is dense in $S^1$, and $S^1$ is compact.
        For each $i$, choose an oriented closed curve $\alpha_i$ in $S$ based at the point $s$ such that the holonomy around $\alpha_i$ is $\varphi [\alpha_i] = \gamma_i$, where $\varphi : \pi_1 (S, s) \to \Gamma$ is the holonomy representation induced by the pointed universal cover $p : (\HH^2, \tilde s) \to (S, s)$.
        Since the set $\bigcup_{i = 1}^n \alpha_i$ is a compact set, the continuous function $\mathup{inj}_S$ has a minimum value $\delta > 0$ on it.
        We define $V_1$ to be the set
        \[
        V_1 = \left\{\psi \in \Homeo^+ (S) \middle | d_S(\psi(q), q) < \delta \text{ for all } q \in \bigcup_{i=1}^n \alpha_i \right\}
        \]
        which is a basic open neighbourhood of $\mathup{id}_S$ of $\Homeo^+ (S)$ in the topology of uniform convergence on compact sets.
        
        Now let $\psi \in V_1$.
        Since $d_S (\psi (s), s) < \delta \le \mathup{inj}_S (s)$, there is exactly one lift of $\psi (s)$ in the $\delta$-neighbourhood of the lift $\tilde s$ of $s$.
        Define $\tilde \psi$ to be the unique lift of $\psi$ that maps $\tilde s$ to this point.
        Further, for each $i$ and each $q \in \alpha_i$, since $d_S (\psi (q), q) < \delta \le \mathup{inj}_S (q)$, there is a unique shortest geodesic from $q$ to $\psi (q)$.
        Moving every point $q$ along this geodesic at constant speed yields a homotopy $H$ of the curve $\alpha_i$ to the curve $\psi (\alpha_i)$.
        Lifting the homotopy to universal covers, we obtain a homotopy $\tilde H$ of the (bi-infinite) lift $\tilde \alpha_i$ of $\alpha_i$ passing through $\tilde s$ to the lift $\tilde \psi (\tilde \alpha_i)$ of $\psi (\alpha_i)$ passing through $\tilde \psi (\tilde s)$.
        See Figure 6. 
        If $\tilde \beta$ is the track of the point $\tilde s$ under this homotopy, and $\beta = p (\tilde \beta)$ is the projection of the curve $\tilde \beta$ to $S$, then it is clear that $[\psi (\alpha_i)] \in \pi_1 (S, \psi (s))$ equals $[\beta \cdot \alpha_i \cdot \overline \beta] = c_\beta [\alpha_i]$.
        If $\varphi' : \pi_1 (S, \psi (s)) \to \Gamma$ is the holonomy representation induced by the pointed universal cover $p : (\HH^2, \tilde \psi (\tilde s)) \to (S, \psi (s))$, then we have $\varphi' [\psi (\alpha_i)] = \varphi' (c_\beta (\alpha_i)) = \varphi [\alpha] = \gamma_i$.
        That is, the holonomy around $\psi (\alpha_i)$ equals the holonomy around $\alpha_i$, which is $\gamma_i$.
        However, the holonomy around $\psi (\alpha_i)$ is $\psi_* (\gamma_i)$.
        We infer that $\psi_* (\gamma_i) = \gamma_i$.
        Consequently, $\partial \tilde \psi ((\gamma_i)_\infty) = (\psi_* (\gamma_i))_\infty = (\gamma_i)_\infty$.
        Since $i$ was arbitrary, this is true for all $i$.
        
        Next, we show the membership $\partial \tilde \psi \in B_d (\mathup{id}_{S^1}, \varepsilon)$ as promised.
        For any $q \in S^1$, if $q = (\gamma_i)_\infty$ for some $i$, then trivially $d (\partial \tilde \psi (q), q) = 0 < \varepsilon$.
        Otherwise, $q$ lies in an interval component of $S^1 \setminus \{(\gamma_1)_\infty, (\gamma_2)_\infty, \ldots, (\gamma_n)_\infty\}$ bounded by two sinks $(\gamma_i)_\infty$ and $(\gamma_j)_\infty$.
        Since the set of sinks is $\frac\varepsilon 2$-dense, the length of this interval is at most $\varepsilon$.
        Since the homeomorphism $\partial \tilde \psi$ preserves the circular order on $S^1$, we conclude that $\partial \tilde \psi (q)$ lies in the same interval as does $q$.
        Hence $d (\partial \tilde \psi (q), q) < \varepsilon$, as both lie in an interval of length less than $\varepsilon$.
        Thus $d (\partial \tilde \psi (q), \mathup{id}_{S^1} (q)) < \varepsilon$ for all $q \in S^1$, and so $\partial \tilde \psi \in B_d (\mathup{id}_{S^1}, \varepsilon)$.
        This concludes the proof of continuity of $\Psi_p$.
        
        Finally, we show that $\Psi_p^{-1} : \mathup N (\Gamma)/\Gamma \to \MCG (S)$ is continuous by describing the inverse explicitly.
        Let $\mathup{DE} : \Homeo^+ (S^1) \to \Homeo^+ (\HH^2)$ denote the Douady-Earle extension.
        $\mathup{DE}$ is a continuous function, where $\Homeo^+ (\HH^2)$ has the compact-open topology.
        Suppose $G \in \mathup N (\Gamma) \subset \Homeo^+ (S^1)$ and $\gamma \in \Gamma \subset \mathup{PSL} (2, \RR)$.
        Then since $G$ normalises $\Gamma$ and $\mathup{DE}$ is conformally natural, we have $\mathup{DE} (G) \circ \gamma = \mathup{DE} (G \circ \gamma) = \mathup{DE} (\gamma' \circ G) = \gamma' \circ \mathup{DE} (G)$ for some $\gamma' \in \Gamma$.
        Thus $\mathup{DE}$ maps $\mathup N (\Gamma)/\Gamma$ into $\Homeo^+_\Gamma (\HH^2)$, which is the set of homeomorphisms of $\HH^2$ that are lifts of homeomorphisms of $S$.
        Let $\pi$ denote the function $\Homeo^+_\Gamma (\HH^2) \to \Homeo^+ (S)$, which assigns the homeomorphism of $S$ to each of its lifts.
        
        We show that $\pi$ is continuous.
        Assume that $K \subset S$ is a compact set and $U \subset S$ is an open set, and consider the open set $B (K, U) = \{\psi \in \Homeo^+ (S) \vert \psi (K) \subset U\}$ in the compact-open topology in $\Homeo^+ (S)$.
        Then there exists a large enough closed ball $D \subset \HH^2$ such that $K$ is contained in $p (D)$.
        Also, $p^{-1} (U) \subset \HH^2$ is open.
        Let $B (D, p^{-1} (U))$ be the set $\left\{\phi \in \Homeo^+ (\HH^2) \mid \phi (D) \subset p^{-1} (U)\right\}$.
        Then for any $\phi \in B (D, p^{-1} (U))$ and any $s \in K \subset p (D)$ with lift $\tilde s \in D$, we have
        $\pi (\phi) (s) = p (\phi (\tilde s)) \in p (p^{-1} (U)) = U$.
        Thus $\phi (B (D, p^{-1} (U))) \subset B (K, U)$ and so $\pi$ is continuous.
        
        It follows that the composition $\pi \circ \mathup{DE}$ is continuous.
        Each $\gamma \in \Gamma$ is a deck transformation of the universal cover $p : \HH^2 \to S$.
        Hence for each $G \in \mathup N (\Gamma)$ and $\gamma \in \Gamma$, we have $\pi \circ \mathup{DE} (G \circ \gamma) = \pi (\mathup{DE} (G) \circ \gamma) = \pi \circ \mathup{DE} (G)$, and so $\pi \circ \mathup{DE}$ descends to a continuous function $\overline{\pi \circ \mathup{DE}} : \mathup N (\Gamma)/\Gamma \to \MCG (S)$.
        For $[G] \in \mathup N (\Gamma)/\Gamma$, a lift of $\overline{\pi \circ \mathup{DE}} [G]$ to $\HH^2$ is simply $\mathup{DE} (G)$, which extends at infinity to $G$, so $\Psi_p \circ \overline{\pi \circ \mathup{DE}}$ is identity on $\mathup N (\Gamma)/\Gamma$.
        Hence $\overline{\pi \circ \mathup{DE}}$ is a right inverse to $\Psi_p$.
        But we already know that $\Psi_p$ is a bijection, so in fact $\overline{\pi \circ \mathup{DE}} = \Psi_p^{-1}$.
        Therefore $\Psi_p^{-1}$ is continuous, and $\Psi_p$ is a homeomorphism.
    \end{proof}
    
    We are ready to prove the continuity of the action function $A$.
    The thrust of the argument is already proved in Lemma \ref{lem:mcg_normaliser}, we only need to finish up the point set topological details.
    \begin{theorem}
        \label{thm:action_continuous}
        The action function $A : \MCG (S) \times \mathcal T (S) \to \mathcal T (S)$ given by equation (\ref{eq:action}) is continuous.
    \end{theorem}
    \begin{proof}
        Fix a universal cover $p : \HH^2 \to S$ with deck group $\Gamma \subset \mathup{PSL} (2, \RR)$, a torsion-free Fuchsian group of the first kind.
        Refer to the commutative diagram in Figure 7. 
        \begin{figure} \label{fig:continuity_point_set_topology}
            \tikzcdset{column sep/huger/.initial=12em}
            \tikzcdset{every label/.append style = {font = \small}}
            $\begin{tikzcd}[column sep = huger, row sep = large]
                \Homeo^+ (S^1) \times \Homeo^+ (S^1) \arrow [r, "{(G, F) \mapsto F \circ G^{-1}}"] & \Homeo^+ (S^1) \\
                \mathup N (\Gamma) \times \widetilde {\mathcal T} (p) \arrow [r, "{(G, F) \mapsto F \circ G^{-1}}"] \arrow [u, hook] \arrow [d, two heads, "{\pi_N \times \pi_p}"'] & \widetilde {\mathcal T} (p) \arrow [u, hook] \arrow [d, two heads, "{\pi_p}"] \\
                \mathup N (\Gamma)/\Gamma \times \mathcal T (p) \arrow[r,"{([\partial \tilde \psi], [\partial \tilde f]) \mapsto [(\partial \tilde f) \circ (\partial \tilde \psi)^{-1}]}"] & \mathcal T (p) \\
                \MCG (S) \times \mathcal T (S) \arrow[r, "{([\psi], [ X, f]) \mapsto [X, f \circ \psi^{-1}]}", "A"'] \arrow [u, "{\Psi_p \times \Phi_p}", "{\cong}"'] & \mathcal T (S) \arrow[u, leftarrow, "{\Phi_p^{-1}}"', "{\cong}"]
            \end{tikzcd}$
            \caption{Continuity of the action function $A$.}
        \end{figure}
        The proposition is that the function in the bottom row is continuous.
        This is equivalent to the continuity of the function in the row just above it (the third row), because $\Phi_p$ is a homeomorphism by definition, and $\Psi_p$ is a homeomorphism by Lemma \ref{lem:mcg_normaliser}.
        The function in the top row is certainly continuous, because multiplication and inversion are continuous functions of the topological group $\Homeo^+ (S^1)$.
        Restricting to subspaces in the second row, we see that the function in the second row is continuous as well.
        The quotient maps $\pi_N : \mathup N (\Gamma) \to \mathup N (\Gamma)/\Gamma$ and $\pi_p : \widetilde {\mathcal T} (p) \to \mathcal T (p)$ are open, because they are quotients by the group actions of $\Gamma$ and $\mathup{PSL} (2, \RR)$ respectively.
        Therefore the product map $\pi_N \times \pi_p$ is also open and surjective, hence a quotient map.
        By the universal property of quotients, the map in the second row descends to the function in the third row which is therefore continuous, as required.
        Thus the action function $A$ is continuous.
    \end{proof}
    
    \begin{cor} \label{cor:action_by_homeomorphisms}
        The mapping class group acts on the marked moduli space by homeomorphisms.
    \end{cor}
    \begin{proof}
        For each fixed mapping class $[\psi] \in \MCG (S)$, its action on the marked moduli space is given by  the function $A_{[\psi]} : \mathcal T (S) \to \mathcal T (S)$ which is $A_{[\psi]} [X, f] = A ([\psi], [X, f])$ and is continuous by Theorem \ref{thm:action_continuous}.
        $A_{[\psi^{-1}]}$ is evidently a continuous inverse, so $A_{[\psi]}$ is a homeomorphism.
    \end{proof}
    
    \section{Embedding the Marked Moduli Space into the Character Space}
    \label{sec:embedding}
    
    In this section, we prove that the marked moduli space $\mathcal T (S)$ embeds into the character space $X (\pi_1 (S), \mathup{PSL} (2, \RR)) = \operatorname{Hom} (\pi_1 (S), \mathup{PSL} (2, \RR))/\mathup{PSL} (2, \RR)$.
    This will imply that in the case of finite type surfaces $S$, that is, if the Euler characteristic $\chi (S)$ is finite, the topology on $\mathcal T (S)$ coincides with the usual topology on Teichm\"uller space.
    Here the representation space $\operatorname{Hom} (\pi_1 (S, s), \mathup{PSL} (2, \RR))$ has the topology of pointwise convergence, and the character space $X (\pi_1 (S, s), \mathup{PSL} (2, \RR))$ has the quotient topology.
    
    The marked moduli space injects into the character space in the obvious way via the character of the holonomy representation.
    We describe it as follows.
    Let $s \in S$ be a basepoint on $S$.
    For a marked hyperbolic structure $[X, f] \in \mathcal T (S)$, we choose the basepoint on $X$ to be $x = f (s)$.
    The marking homeomorphism $f : S \to X$ induces an isomorphism of fundamental groups $f_* : \pi_1 (S, s) \to \pi_1 (X, x)$.
    Now choose a pointed universal cover $p_X : (\HH^2, \tilde x) \to (X, x)$ with deck group $\Gamma_X \subset \mathup{PSL} (2, \RR)$.
    We have a holonomy representation $\varphi_X : \pi_1 (X, x) \xrightarrow{\cong} \Gamma_X \hookrightarrow \mathup{PSL} (2, \RR)$.
    This leads to the representation $\rho_{[X, f]} = \varphi_X \circ f_* : \pi_1 (S, s) \to \mathup{PSL} (2, \RR)$.
    
    \begin{definition}
        We define $\Phi_{\mathup{at}} [X, f]$ to be the character of $\rho_{[X, f]}$ in $X (\pi_1 (S, s), \mathup{PSL} (2, \RR))$.
    \end{definition}
    
    \begin{prop}
        $\Phi_{\mathup{at}} : \mathcal T (S) \to X (\pi_1 (S, s), \mathup{PSL} (2, \RR))$ is a well defined and injective function.
    \end{prop}
    \begin{proof}
        We have to show that the character $[\rho_{[X, f]}]$ is independent of the choice of the pointed universal cover $p_X$ and independent of the representative $(X, f)$ of the marked hyperbolic structure.
        A different pointed universal cover is of the form $p_X \circ \sigma^{-1} : (\HH^2, \sigma (\tilde x)) \to (X, x)$ for some $\sigma \in \mathup{PSL} (2, \RR)$.
        Let $\rho_{[X, f]}'$ be the representation obtained using the new pointed universal cover.
        Suppose $\gamma$ is a loop in $S$ based at $s$, so that $[\gamma] \in \pi_1 (S, s)$.
        Suppose that the lift of $f (\gamma)$ to the universal cover (with respect the covering map $p_X$) starting at $\tilde x$ has endpoint $\tilde x'$.
        Then the lift, with respect to the covering map $p_X \circ \sigma^{-1}$, starting at $\sigma (\tilde x)$ has endpoint $\sigma (\tilde x')$.
        Consequently, the deck transformation sending $\sigma (\tilde x)$ to $\sigma (\tilde x')$ is $\sigma \circ \varphi_X (f_* [\gamma]) \circ \sigma^{-1}$.
        But this deck transformation is precisely the holonomy $\rho'_{[X, f]} [\gamma]$ with respect to the new pointed universal cover, which we are seeking.
        In other words, it is the conjugate of $\rho_{[X, f]} [\gamma]$ by $\sigma \in \mathup{PSL} (2, \RR)$, for all $[\gamma] \in \pi_1 (S, s)$.
        So we conclude that the representation $\rho_{[X, f]}'$ is conjugate to $\rho_{[X, f]}$ by some element of $\mathup{PSL} (2, \RR)$, as was to be shown.
        Thus the character $[\rho_{[X, f]}]$ is independent of the pointed universal cover $p_X$.
        
        Next, we show that the character $[\rho_{[X, f]}]$ is independent of the choice of the marking map $f$ in its homotopy class.
        Let $f'$ be another marking map homotopic to $f$.
        Suppose that the track of the basepoint $s$ under this homotopy is a curve $\beta$, with endpoint $x' = f' (s)$.
        Then we have the relation $f_* = c_\beta \circ f_*'$ between the induced maps on fundamental groups with basepoints $x$ and $x'$.
        Recall that here $c_\beta [\gamma] = [\beta \cdot \gamma \cdot \overline \beta]$ for all $[\gamma] \in \pi_1 (X, x')$.
        Lifting $\beta$ to the universal cover of $X$, we get a curve $\tilde \beta$ which starts at $\tilde x$ (above $x$), and ends at $\tilde x'$ (above $x'$).
        The pointed universal cover $p_X : (\HH^2, \tilde x') \to (X, x')$ yields a new holonomy representation $\varphi_X' : \pi_1 (X, x') \to \Gamma_X \hookrightarrow \mathup{PSL} (2, \RR)$ which satisfies the relation $\varphi_X' = \varphi_X \circ c_\beta$.
        Thus the new representation is $\rho_{[X, f]}' = \varphi_X' \circ f_*'$, which equals $\varphi_X' \circ f_*' = (\varphi_X \circ c_\beta) \circ f_*' = \varphi_X \circ (c_\beta \circ f_*') = \varphi_X \circ f_* = \rho_{[X, f]}$.
        Thus the representation, and therefore its character, is independent of the choice of the marking map in its homotopy class.
        
        To show that the character $[\rho_{[X, f]}]$ does not depend on the representative $(X, f)$ of the marked hyperbolic structure, suppose that $(Y, g)$ is another representative.
        This means that there exists an isometry $\psi : X \to Y$ such that $g$ is homotopic to $\psi \circ f$.
        We have already shown that homotopic marking maps yield the same holonomy representation, so we may assume that $g = \psi \circ f$.
        The isometry $\psi$ lifts to an isometry of universal covers $\sigma = \tilde \psi \in \mathup{PSL} (2, \RR)$.
        As before, the choice of a pointed universal cover $p_Y : (\HH^2, \tilde y) \to (Y, y)$, where $y = g (s)$, leads to a holonomy representation $\varphi_Y : \pi_1 (Y, y) \to \Gamma_Y \hookrightarrow \mathup{PSL} (2, \RR)$.
        For every $[\gamma] \in \pi_1 (S, s)$, we have $\rho_{[Y, g]} [\gamma] = \varphi_Y (g_* [\gamma]) = \varphi_Y ((\psi \circ f)_* [\gamma]) = \varphi_Y (\psi_* (f_* [\gamma])) = \psi_* (\varphi_X (f_* [\gamma])) = \psi_* (\rho_{[X, f]} [\gamma])$.
        However $\psi_* : \Gamma_X \to \Gamma_Y$ is conjugation by $\tilde \psi = \sigma$, so $\rho_{[Y, g]} [\gamma] = \psi_* (\rho_{[X, f]} [\gamma]) = \sigma \circ \rho_{[X, f]} [\gamma] \circ \sigma^{-1}$.
        In words, $\rho_{[Y, g]} [\gamma]$ is the conjugate of $\rho_{[X, f]} [\gamma]$ by $\sigma \in \mathup{PSL} (2, \RR)$.
        This holds for all $[\gamma] \in \pi_1 (S, s)$, so we conclude that the representation $\rho_{[Y, g]}$ is conjugate to $\rho_{[X, f]}$ by some element of $\mathup{PSL} (2, \RR)$.
        We conclude that the character $[\rho_{[X, f]}]$ is independent of the choice of representative $(X, f)$ of the marked hyperbolic structure.
        Therefore $\Phi_{\mathup{at}}$ is a well defined function.
        
        Finally, we show that $\Phi_{\mathup{at}}$ is injective.
        Suppose $[X, f], [Y, g] \in \mathcal T (S)$ are marked hyperbolic structures such that $[\rho_{[X, f]}] = [\rho_{[Y, g]}]$.
        This means that there is an element $\sigma \in \mathup{PSL} (2, \RR)$ such that $\sigma \circ \rho_{[X, f]} [\gamma] \circ \sigma^{-1} = \rho_{[Y, g]} [\gamma]$ for every $[\gamma] \in \pi_1 (S, s)$.
        Thus $\sigma \circ \rho_{[X, f]} [\gamma] = \rho_{[Y, g]} [\gamma] \circ \sigma$, so $\sigma$ is $\pi_1 (S, s)$-equivariant.
        Hence $\sigma$ descends to an isometry $\psi : X \to Y$, and we have $(\psi \circ f)_* = g_*$ as homomorphisms $\pi_1 (S, s) \to \pi_1 (Y, y)$.
        Since the surface $S$ is a $K (\pi_1 (S, s), 1)$ classifying space, we conclude that $\psi \circ f$ is homotopic to $g$.
        Therefore $[X, f] = [Y, g]$, and $\Phi_{\mathup{at}}$ is injective.
    \end{proof}
    
    Since the topology on $\mathcal T (S)$ is defined using a fixed universal cover $p$ and the function $\Phi_p$, in order to prove $\Phi_{\mathup{at}}$ is an embedding, we must compute $\Phi_{\mathup{at}}$ in terms of $\Phi_p$.
    To that end, fix a pointed universal cover $p: (\HH^2, \tilde s) \to (S, s)$ with deck group $\Gamma \subset \mathup{PSL} (2, \RR)$, which is a torsion-free Fuchsian group of the first kind.
    This leads to a holonomy representation $\rho_p = \varphi : \pi_1 (S, s) \to \Gamma \hookrightarrow \mathup{PSL} (2, \RR)$.
    For any marked hyperbolic structure $[X, f]$ with a pointed universal cover $p_X : (\HH^2, \tilde x) \to (X, x)$, where $x = f (s)$, the marking map $f$ lifts to a map $\tilde f$, which extends, by Proposition \ref{prop:homeomorphism_at_infinity} (\ref{prop:homeomorphism_at_infinity_extends}), to a homeomorphism at infinity $\partial \tilde f$.
    As before, we have the representation $\rho_{[X, f]} = \varphi_X \circ f_*$.
    This is related to $\partial \tilde f$ as follows.
    \begin{prop} \label{prop:homeomorphism_to_representation}
        For any $[\gamma] \in \pi_1 (S, s)$, we have
        \begin{equation}\label{eq:homeomorphism_to_representation_specific}
            \rho_{[X, f]} [\gamma] = (\partial \tilde f) \circ \rho_p [\gamma] \circ (\partial \tilde f)^{-1}
        \end{equation}
        In general, if $\Phi_p [X, f]$ is represented by the homeomorphism $F$, then $\Phi_{\mathup{at}} [X, f]$ is represented by the representation $R$ which satisfies, for all $[\gamma] \in \pi_1 (S)$,
        \begin{equation}\label{eq:homeomorphism_to_representation_general}
            R [\gamma] = F \circ \rho_p [\gamma] \circ F^{-1}
        \end{equation}
    \end{prop}
    \begin{proof}
        $\rho_{[X, f]} [\gamma] = \varphi_X (f_* [\gamma]) = f_* (\varphi [\gamma]) = \tilde f \circ \varphi [\gamma] \circ \tilde f^{-1} = \tilde f \circ \rho_p [\gamma] \circ \tilde f^{-1}$.
        Note that in the second equality, the $f_*$ on the left hand side is the $\pi_1$ functor, whereas on the right hand side, $f_*$ is the isomorphism between the deck groups $\Gamma$ and $\Gamma_X$ induced by $\tilde f$.
        Extending all maps at infinity using Proposition \ref{prop:homeomorphism_at_infinity}, parts (\ref{prop:homeomorphism_at_infinity_composition}) and (\ref{prop:homeomorphism_at_infinity_isometrycomp}), we get $\rho_{[X, f]} [\gamma] = (\partial \tilde f) \circ \rho_p [\gamma] \circ (\partial \tilde f)^{-1}$.
        Any other homeomorphism $F$ representing $\Phi_p [X, f]$ is of the form $\sigma \circ (\partial \tilde f)$ for some $\sigma \in \mathup{PSL} (2, \RR)$.
        Then for all $[\gamma] \in \pi_1 (S)$, we have $\sigma \circ \rho_{[X, f]} \circ \sigma^{-1} = \sigma \circ (\partial \tilde f) \circ \rho_p [\gamma] \circ (\partial \tilde f)^{-1} \circ \sigma^{-1} = (\sigma \circ (\partial \tilde f)) \circ \rho_p [\gamma] \circ (\sigma \circ (\partial \tilde f))^{-1} = F \circ \rho_p [\gamma] \circ F^{-1} = R [\gamma]$.
        Thus the representation $R$ is conjugate to the representation $\rho_{[X, f]}$, and hence the characters of $\rho_{[X, f]}$ and $R$ are equal.
        That is, $R$ represents $\Phi_{\mathup{at}} [X, f]$ and satisfies equation (\ref{eq:homeomorphism_to_representation_general}).
    \end{proof}
    
    We can also compute $\Phi_p$ in terms of $\Phi_{\mathup{at}}$.
    \begin{prop} \label{prop:representation_to_homeomorphism}
        For any $[\gamma] \in \pi_1 (S, s)$, we have
        \begin{equation} \label{eq:representation_to_homeomorphism_specific}
            \partial \tilde f ((\rho_p [\gamma])_\infty) = (\rho_{[X, f]} [\gamma])_\infty
        \end{equation}
        If general, if $\Phi_{\mathup{at}} [X, f]$ is represented by the representation $R$, then $\Phi_p [X, f]$ is represented by the homeomorphism $F$ which satisfies, for all $[\gamma] \in \pi_1 (S, s)$,
        \begin{equation}\label{eq:representation_to_homeomorphism_general}
            F ((\rho_p [\gamma])_\infty) = (R [\gamma])_\infty
        \end{equation}
    \end{prop}
    \begin{proof}
        To establish equation (\ref{eq:representation_to_homeomorphism_specific}), we use the proof of Proposition \ref{prop:homeomorphism_at_infinity}: we have $\partial \tilde f ((\rho_p [\gamma])_\infty) = (f_* (\rho_p [\gamma]))_\infty = (f_* (\varphi [\gamma]))_\infty = (\varphi_X (f_* [\gamma]))_\infty = (\rho_{[X, f]} [\gamma])_\infty$.
        Note that in the third equality, the $f_*$ on the left hand side is the isomorphism between the deck group $\Gamma$ and $\Gamma_X$ induced by $\tilde f$, whereas on the right hand side, $f_*$ is the $\pi_1$ functor.
        In the general case, since $\Phi_{\mathup{at}} [X, f]$ is represented by $R$, it follows that $R$ is conjugate to the representation $\rho_{[X, f]}$ by some $\sigma \in \mathup{PSL} (2, \RR)$.
        That is, for every $[\gamma] \in \pi_1 (S, s)$, we have $R [\gamma] = \sigma \circ \rho_{[X, f]} [\gamma] \circ \sigma^{-1}$.
        Therefore  the sink $(R [\gamma])_\infty$ is just the image under $\sigma$ of the sink $(\rho_{[X, f]} [\gamma])_\infty$.
        In other words, $(R [\gamma])_\infty = \sigma ((\rho_{[X, f]} [\gamma])_\infty)$, which equals, by equation (\ref{eq:representation_to_homeomorphism_specific}), $\sigma \circ (\partial \tilde f) ((\rho_p [\gamma])_\infty)$.
        Setting $F = \sigma \circ (\partial \tilde f)$, we see that $F$ also represents the same right coset of $\mathup{PSL} (2, \RR)$ in $\Homeo^+ (S^1)$ as that of $(\partial \tilde f)$, and hence represents $\Phi_p [X, f]$.
        Thus we have established equation (\ref{eq:representation_to_homeomorphism_general}).
    \end{proof}
    
    \begin{theorem}
        \label{thm:embedding_character_space}
        $\Phi_{\mathup{at}} : \mathcal T (S) \to X (\pi_1 (S, s), \mathup{PSL} (2, \RR))$ is a topological embedding.
    \end{theorem}
    \begin{proof}
        We need to show that a sequence of marked hyperbolic structures $[X_i, f_i]$ converges to $[X, f]$ as $i \to \infty$ if and only if the corresponding sequence of characters $\Phi_{\mathup{at}} [X_i, f_i]$ converges to $\Phi_{\mathup{at}} [X, f]$ as $i \to \infty$.
        
        We begin with the `only if' part.
        Suppose $[X_i, f_i] \to [X, f]$ as $i \to \infty$ in $\mathcal T (S)$.
        Since $\Phi_p$ is a homeomorphism, we have $\Phi_p [X_i, f_i] \to \Phi_p [X, f]$ in $\mathcal T (p)$.
        That is, $[\partial \tilde f_i] \to [\partial \tilde f]$ in $\mathcal T (p)$ as $i \to \infty$.
        We can promote this to convergence in $\widetilde {\mathcal T} (p)$ by constructing a continuous section $\Sigma_1 : \mathcal T (p) \to \widetilde{\mathcal T} (p)$.
        For any $[\hat F] \in \mathcal T (p)$, let $\Sigma_1 ([\hat F])$ be the unique homeomorphism of $S^1$ in the right coset $\mathup{PSL} (2, \RR) \circ \hat F$ that fixes the three points $0, 1, \infty$.
        In fact, $\Sigma_1 (\hat F) = M (\hat F (0), \hat F (1), \hat F (\infty))^{-1} \circ \hat F$, where $M (a, b, c)$ is the M\"obius transformation mapping the triple $(0, 1, \infty)$ to the triple $(a, b, c)$.
        Since $M (a, b, c)$ is a continuous function of $a, b, c \in S^1$, evaluations at $0, 1, \infty$ are continuous functions of $\hat F$, and compositions are continuous, we conclude that $\Sigma_1$ is a continuous function.
        Therefore $\Sigma_1 ([\partial \tilde f_i]) \to \Sigma_1 ([\partial \tilde f])$ as $i \to \infty$.
        Denoting $\Sigma_1 ([\partial \tilde f_i])$ by $F_i$ and $\Sigma_1 ([\partial \tilde f])$ by $F$, we have $F_i \to F$ in $\widetilde{\mathcal T} (p)$ as $i \to \infty$.
        Note that, by Proposition \ref{prop:homeomorphism_to_representation}, $\Phi_{\mathup{at}} [X, f]$ is represented by the representation $\rho$ which satisfies, for each $[\gamma] \in \pi_1 (S, s)$, the relation $\rho [\gamma] = F \circ \rho_p [\gamma] \circ F^{-1}$.
        Similarly, $\Phi_{\mathup{at}} [X_i, f_i]$ is represented by the representation $\rho_i$ which satisfies, for each $[\gamma] \in \pi_1 (S, s)$, the relation $\rho_i [\gamma] = (F_i) \circ \rho_p [\gamma] \circ (F_i)^{-1}$.
        This $\rho_i$ converges, as $i \to \infty$, to $(F) \circ \rho_p [\gamma] \circ (F)^{-1} = \rho [\gamma]$.
        On the level of representations, this means $\rho_i \to \rho$ and on the level of characters, $[\rho_i] \to [\rho]$ as $i \to \infty$.
        In other words, $\Phi_{\mathup{at}} [X_i, f_i] \to \Phi_{\mathup{at}} [X, f]$ as $i \to \infty$ and therefore $\Phi_{\mathup{at}}$ is continuous.
        
        Next we prove the `if' part.
        Suppose $\Phi_{\mathup{at}} [X_i, f_i] \to \Phi_{\mathup{at}} [X, f]$ as $i \to \infty$.
        We promote the convergence of characters to convergence of representations by constructing a continuous section $\Sigma_2 : \Phi_{\mathup{at}} (\mathcal T (S)) \to \operatorname{Hom} (\pi_1 (S, s), \mathup{PSL} (2, \RR))$.
        Let $\alpha_1, \alpha_2, \alpha_3 \in \pi_1 (S, s)$ be such that for any $[\hat \rho] \in \Phi_{\mathup{at}} (\mathcal T (S))$, we have $\hat \rho (\alpha_k)$ is a hyperbolic element of $\mathup{PSL} (2, \RR)$ for each $k = 1, 2, 3$.
        We can choose these elements $\alpha_1, \alpha_2, \alpha_3$ such that the sinks $(\hat \rho (\alpha_k))_\infty$, $k = 1, 2, 3$, are distinct, since the set $(\Gamma_{\hat X})_\infty$ of sinks is dense in $S^1$ for any $[\hat X, \hat f] \in \mathcal T (S)$.
        Relabelling if necessary, we assume that $((\hat \rho (\alpha_1))_\infty, (\hat \rho (\alpha_2))_\infty, (\hat \rho (\alpha_3))_\infty)$ is a positively oriented triple of points in $S^1$.
        Hyperbolicity of an element of $\pi_1 (S, s)$ does not depend on the chosen marked hyperbolic structure or the particular representation in its conjugacy class, and neither does the circular order of the three points $(\hat \rho (\alpha_1))_\infty, (\hat \rho (\alpha_2))_\infty, (\hat \rho (\alpha_3))_\infty$.
        Now, for any $\hat \rho \in \Phi_{\mathup{at}} (\mathcal T (S))$, let $\Sigma_2 [\hat \rho]$ be the unique representation that represents the same character as $\hat \rho$ and such that the sinks of $\Sigma_2 [\hat \rho] (\alpha_1), \Sigma_2 [\hat \rho] (\alpha_2), \Sigma_2 [\hat \rho] (\alpha_3)$ are $0, 1, \infty$ respectively.
        In fact $\Sigma_2 [\hat \rho]$ is the conjugate of $\hat \rho$ by the M\"obius transformation $M ((\hat \rho_i (\alpha_1))_\infty, (\hat \rho_i (\alpha_2))_\infty, (\hat \rho_i (\alpha_3))_\infty)^{-1}$.
        Since $M (a, b, c)$ is a continuous function of $a, b, c \in S^1$, evaluations at $\alpha_1, \alpha_2, \alpha_3 \in \pi_1 (S, s)$ are continuous functions of the representation $\hat \rho$, the sink of a hyperbolic isometry is a continuous function of the hyperbolic isometry, and compositions are continuous, we conclude that $\Sigma_2$ is continuous.
        Therefore $\Sigma_2 (\Phi_{\mathup{at}} [X_i, f_i]) \to \Sigma_2 (\Phi_{\mathup{at}} [X, f])$ as $i \to \infty$.
        Denoting $\Sigma_2 (\Phi_{\mathup{at}} [X_i, f_i])$ by $\rho_i$ and $\Sigma_2 (\Phi_{\mathup{at}} [X, f])$ by $\rho$, we have $\rho_i \to \rho$ in $\operatorname{Hom} (\pi_1 (S, s), \mathup{PSL} (2, \RR))$ as $i \to \infty$.
        This means that for all hyperbolic elements $[\gamma] \in \pi_1 (S, s)$, we have $\rho_i [\gamma] \to \rho [\gamma]$, and so $(\rho_i [\gamma])_\infty \to (\rho [\gamma])_\infty$ as $i \to \infty$.
        
        Note that, by Proposition \ref{prop:representation_to_homeomorphism}, $\Phi_p [X, f]$ is represented by the homeomorphism $F$ of $S^1$, which satisfies, for every hyperbolic $[\gamma] \in \pi_1 (S, s)$, the relation $F ((\rho_p [\gamma])_\infty) = (\rho [\gamma])_\infty$.
        Similarly $\Phi_p [X_i, f_i]$ is represented by the homeomorphism $F_i \in \Homeo^+ (S^1)$ which satisfies, for every hyperbolic $[\gamma] \in \pi_1 (S, s)$, the relation $F_i ((\rho_p [\gamma])_\infty) = (\rho_i [\gamma])_\infty$.
        To show that $[X_i, f_i]$ converges to $[X, f]$, it is enough to prove that $F_i$ converges to $F$ as $i \to \infty$, in the topology of uniform convergence on $S^1$ with respect to some metric $d$.
        Let $\varepsilon > 0$ be given.
        Reducing $\varepsilon$ if necessary, assume $\varepsilon < \frac 12$.
        Choose hyperbolic elements $[\gamma_1], [\gamma_2], \ldots, [\gamma_n] \in \pi_1 (S, s)$ such that their sinks $(\rho [\gamma_1])_\infty, (\rho [\gamma_2])_\infty, \ldots, (\rho [\gamma_n])_\infty$ form an $\frac{\varepsilon}{12}$-dense set in $S^1$.
        This is possible because the set $\Gamma_X$ of sinks is dense in $S^1$ and $S^1$ is compact.
        Relabelling if necessary, we assume that the points $(\rho [\gamma_1])_\infty, (\rho [\gamma_2])_\infty, \ldots, (\rho [\gamma_n])_\infty$ are in positive circular order.
        These points divide the circle at infinity into $n$ intervals, each of length at most $\frac\varepsilon 6$.
        For each $j = 1, 2, \ldots, n$, we have $(\rho_p [\gamma_j])_\infty = F^{-1} (\rho [\gamma_j])_\infty$.
        As $F^{-1}$ is an orientation preserving homeomorphism, it follows that the sinks $(\rho_p [\gamma_1])_\infty, (\rho_p [\gamma_2])_\infty, \ldots, (\rho_p [\gamma_n])_\infty$ are in positive circular order.
        Similarly for each $i$ and each $j = 1, 2, \ldots, n$, we have $(\rho_i [\gamma_j])_\infty = F_i ((\rho_p [\gamma_j])_\infty)$, and since $F_i$ is an orientation preserving homeomorphism, it follows that the points $(\rho_i [\gamma_1])_\infty, (\rho_i [\gamma_2])_\infty, \ldots, (\rho_i [\gamma_n])_\infty$ are also in positive circular order.
        Since sinks depends continuously on hyperbolic isometries, for each $j = 1, 2, \ldots, n$, the convergence $\rho_i [\gamma_j] \to \rho [\gamma_j]$ implies $(\rho_i [\gamma_j])_\infty \to (\rho [\gamma_j])_\infty$ as $i \to \infty$.
        Therefore for all sufficiently large $i$ and each of the finitely many indices $j = 1, 2, \ldots, n$, we have $d ((\rho_i [\gamma_j])_\infty, (\rho [\gamma_j])_\infty) < \frac{\varepsilon}{6}$.
        Further, we have $d ((\rho_i [\gamma_j])_\infty, (\rho_i [\gamma_{j + 1}])_\infty) \le
        d ((\rho_i [\gamma_j])_\infty, (\rho [\gamma_j])_\infty) + d ((\rho [\gamma_j])_\infty, (\rho [\gamma_{j + 1}])_\infty) + d ((\rho [\gamma_{j + 1}])_\infty, (\rho_i [\gamma_{j + 1}])_\infty) < \frac \varepsilon 6 + \frac \varepsilon 6 + \frac \varepsilon 6 = \frac \varepsilon 2$.
        Thus the points $(\rho_i [\gamma_1])_\infty, (\rho_i [\gamma_2])_\infty, \ldots, (\rho_i [\gamma_n])_\infty$ divide $S^1$ into intervals of length at most $\frac \varepsilon 2$.
        
        Suppose $q \in S^1$ is an arbitrary point.
        Then for some $j$, $q$ belongs to an interval bounded by points $(\rho_p [\gamma_j])_\infty, (\rho_p [\gamma_{j + 1}])_\infty$.
        Since $F_i$ is an orientation preserving homeomorphism, it follows that $F_i (q)$ is in the interval bounded by the points $(\rho_i [\gamma_j])_\infty$ and $(\rho_i [\gamma_{j + 1}])_\infty$.
        Therefore $d (F_i (q), (\rho_i [\gamma_j])_\infty)$ is less than the length of this interval, which is less than $\frac \varepsilon 2$.
        Similarly, $F$ is an orientation preserving homeomorphism, so $F (q)$ lies in the interval bounded by the points $(\rho [\gamma_j])_\infty$ and $(\rho [\gamma_{j + 1}])_\infty$, and hence $d (F (q), (\rho [\gamma_j])_\infty)$ is less than the length of this interval, which is less than $\frac \varepsilon 6$.
        Now we have $d (F_i (q), F (q)) \le d (F_i (q), (\rho_i [\gamma_j])_\infty) + d ((\rho_i [\gamma_j])_\infty, (\rho [\gamma_j])_\infty) + d ((\rho [\gamma_j])_\infty, F (q)) < \frac \varepsilon 2 + \frac \varepsilon 6 + \frac \varepsilon 6 < \varepsilon$.
        Since $q \in S^1$ was arbitrary and $\varepsilon$ was an arbitrary positive number, we infer that $F_i \to F$ as $i \to \infty$ in the topology of uniform convergence.
        This concludes the `if' part, that is, $\Phi_{\mathup{at}}^{-1} : \Phi_{\mathup{at}} (\mathcal T (S)) \to \mathcal T (S)$ is continuous.
        Hence $\Phi_{\mathup{at}}$ is an embedding.
    \end{proof}
    
    \begin{cor} \label{cor:marked_moduli_space_reduces_to_teichmuller_space}
        If $S$ is a finite type surface of negative Euler characteristic, then the topology on $\mathcal T (S)$ agrees with the topology on Teichm\"uller space.
    \end{cor}
    \begin{proof}
        Indeed, one of the ways to describe the Teichm\"uller space is as the character space of all discrete faithful representations of $\pi_1 (S, s)$ into $\mathup{PSL} (2, \RR)$ (see \cite[Chapter 10]{farb_margalit_2011_primer}).
        In particular, the topology on the Teichm\"uller space is defined as a subspace of the full character space $X (\pi_1 (S, s), \mathup{PSL} (2, \RR))$.
        But this is exactly the topology on $\mathcal T (S)$ that we have defined above.
        Thus $\mathcal T (S)$ reduces to the usual Teichm\"uller space in case $S$ is a finite type surface.
    \end{proof}
    
    \nocite{*}
    \bibliographystyle{alpha}
    \bibliography{marked_moduli_space}
    
\end{document}